\documentclass[twoside,twocolumn]{article}
\usepackage{ltexpprt}
\usepackage{epsfig,epic,eepic,units}
\usepackage{hyperref}
\usepackage{url}
\usepackage{mathrsfs}
\usepackage{multirow}
\usepackage{amssymb}
\usepackage{graphicx}
\usepackage{amsmath}
\usepackage{algorithm,algorithmic}
\usepackage{etoolbox}
\usepackage{color}
\usepackage{wrapfig}
\usepackage[normalem]{ulem}

\newcommand{\bR}{\mathbb{R}}
\newcommand{\eps}{\epsilon}
\newcommand{\veps}{\varepsilon}

\newcommand{\tvec}{\mathtt{vec}}
\newcommand\nc\newcommand
\nc\PhI{\Phi_I}\nc\del{\delta}
\nc\Om{\Omega}

\newcommand{\nnz}{\mathrm{nnz}}
\newcommand{\tr}{\mathtt{tr}}

\newcommand{\mbf}{\mathbf}\newcommand{\bSigma}{\mbf{\Sigma}}
\newcommand{\bOmega}{\mbf{\Omega}}

\title{\Large Find the dimension that counts: Fast dimension estimation and Krylov PCA}
\author{Shashanka Ubaru\thanks{IBM T. J. Watson Research Center, Yorktown Heights, NY, USA. 
 {\tt Shashanka.Ubaru@ibm.com}.}
\and  Abd-Krim Seghouane\thanks{ 
The University of Melbourne, Melbourne, Victoria, Australia.
 {\tt abd-krim.seghouane@unimelb.edu.au}}
\and Yousef Saad\thanks{ 
University of Minnesota, Twin Cities, MN, USA. 
{\tt saad@umn.edu} }}
\date{}
\begin{document}
\maketitle


    \begin{abstract} {\small\baselineskip=9pt  
    High dimensional data and systems with many degrees of freedom are often
    characterized by  covariance matrices.
    In this paper, we consider the problem of simultaneously estimating
    the dimension of the principal (dominant)
    subspace of  these covariance matrices and obtaining an approximation to the subspace.  
    This problem arises in the popular principal component analysis (PCA), and 
    in many applications of machine learning, data analysis,
    signal and image processing, and others.
We first present a novel  method for estimating the dimension of the principal subspace. 
    We then show how this method can be coupled with a Krylov subspace method to 
simultaneously estimate the dimension 
and obtain an approximation to the subspace. The dimension estimation  is achieved at no additional cost.
 The proposed  method operates on a model selection framework,
where the novel selection criterion is derived  
based on random matrix perturbation theory ideas.
We present theoretical analyses which (a) show that the proposed method
achieves strong consistency 
(i.e., yields optimal solution as the number of data-points  $n\rightarrow \infty$),
and (b)  analyze conditions
for exact dimension estimation in the finite $n$ case. 
Using recent results, we show that our algorithm also yields near optimal PCA.
The proposed method avoids forming the sample covariance matrix (associated with the data) explicitly
 and computing the complete eigen-decomposition.
Therefore, the method is inexpensive, which
is particularly advantageous in modern data applications
where the covariance matrices can be very large.
Numerical experiments illustrate the performance of the proposed method in various applications.}
    \end{abstract}
    
    \section{Introduction}
    In many applications, for a given 
    set of data observations, covariance matrices are used to capture the interactions in  high dimensions, among 
    the many degrees of freedom. A popular approach to analyze such high dimensional data is to look for the 
    principal (components) subspace of the covariance matrix, which is of much lower dimension. For this,  it is often 
    required to first estimate the dimension of this principal 
    (dominant)  subspace of the  covariance matrix associated with the 
    observations~\cite{wax1985detection,kritchman2008determining,camba2008statistical,
    kritchman2009non,ubaru2016fast}. 
    These observations can  be treated as high 
    dimensional random quantities embedded in noise.
    
    Low rank approximation is a popular tool used in applications
    to reduce  high dimensional data~\cite{jolliffe2002principal,review,khanna2017deflation,ubaru2017low}. Determining 
    the lower dimension (rank $k$)
    remains a principal problem in these applications, 
    see~\cite{ubaru2016fast,ubaru2017fast} for discussions.
  In statistical signal and array processing,  detecting the number of signals
    in the observations of an array of passive sensors is a fundamental
    problem~\cite{wax1985detection,kritchman2009non,
    nadler2010nonparametric},  which can be posed as the above dimension estimation problem.
    Similar estimation problems occur in many other fields
    such as chemo-metrics~\cite{meloun2000critical,kritchman2008determining}, 
    econometrics and statistics~\cite{camba2008statistical}, 
    population genetics~\cite{patterson2006population}, and 
    reduced rank regression models~\cite{bura2003rank}. 
     Moreover, in most of these applications, once the dimension of the principal 
     subspace (approximate rank) is estimated, 
     it is also desired to obtain an approximation for this 
     principal  subspace, e.g., in principal component analysis (PCA)~\cite{jolliffe2002principal,khanna2017deflation},
     subspace tracking~\cite{comon1990tracking} and others. 
     Krylov subspace based methods~\cite{Saad-book3} are the most 
     popular and effective methods used in the literature to 
     compute an approximation for the principal subspace, 
     see~\cite{xu1994fast2,schneider2001krylov,ide2007change,musco2015randomized,rachakonda2016memory} 
     for examples.
    
    \paragraph{Prior Work:}
    The problem of estimating the rank  or the dimension of the 
    principal  subspace has been studied in various fields, and 
    a few different methods have been proposed in the literature. 
    In signal processing, 
  information theory 
    criteria based methods  have been proposed for
    the detection of number of signals~\cite{wax1985detection,nadler2010nonparametric}. 
    A few hypothesis testing based methods have also been proposed for dimension 
    estimation,
    see~\cite{xu1994fast2,patterson2006population,kritchman2008determining,kritchman2009non}. 
    In econometrics and statistics, various tests and methods have been proposed 
    to estimate the rank and the rank
statistic of a matrix, see, e.g.,~\cite{robin2000tests,donald2007rank,camba2008statistical}. 
 
 However, most of these methods 
    require computing the complete eigen-decomposition
    of the sample covariance matrix, which becomes impractical for 
    large dimensional matrices, e.g., in modern data applications and for large aperture arrays
    in array signal processing.
    Even forming the covariance matrix is not viable in many cases.
    The information criteria based methods are not applicable 
    when the data dimension $p$ is larger than the number of observations $n$,
    i.e., when $p>n$.
    Recently, a set of  inexpensive methods were  proposed for numerical rank estimation
    of data matrices~\cite{ubaru2016fast,ubaru2017fast}.  
   These methods combine ideas such as stochastic trace estimator, eigen-projectors and
   spectral densities to compute the rank inexpensively without any matrix decomposition. 
    However, methods that simultaneously estimate the dimension 
    and obtain an approximation to  the principal 
    subspace are lacking.

    \paragraph{Contributions:}
    In this work, we present a method for estimating the dimension 
    of  the principal  subspace of covariance matrices. 
    The method can be combined with the Krylov subspace methods (Krylov PCA) to compute an approximation 
    to the principal
    subspace, simultaneously. 
    The  method operates on a model selection framework,
and the proposed  selection criterion  requires computing only the top $k$ eigenvalues
    of the sample covariance matrix $\mbf{S}_n=\frac{1}{n}\mbf{XX}^T$, where $\mbf{X}$ 
    is the matrix containing $n$ observed data of dimension $p$, for a
    given integer $k\ll\{n,p\}$.
    In order to compute these eigenvalues, we
     can use the popular Lanczos algorithm~\cite{Saad-book3} which requires only
     matrix-vector products with $\mbf{S}_n$. Hence, we do not have to 
     form the sample covariance matrix $\mbf{S}_n=\frac{1}{n}\mbf{XX}^T$, explicitly. 
     Our approach can be viewed as a stopping criterion for the Krylov subspace 
     methods, and we can simultaneously estimate the dimension and compute the principal
     subspace at no additional cost. 

     The proposed selection criterion is derived using random matrix perturbation theory
     results~\cite{muirhead2009aspects},
     see section~\ref{sec:propose}.
     The criterion also includes a penalty (function) term which under mild assumptions
     yields us a strongly consistent estimator,
     i.e., the method estimates the  exact dimension as the number of data observations
     $n\rightarrow \infty$.
     We establish this strong consistency for the proposed method and also 
present performance analysis in section~\ref{sec:analysis}.
We derive  conditions on the signal strength and the noise level for 
avoiding incorrect dimension estimation in the finite $n$ case,
using random matrix theory results~\cite{johansson2000shape}.
Using the recent 
results in~\cite{musco2015randomized}, we also show that the method yields near optimal PCA, and the consistency results and the performance analysis hold for eigenvalues computed by the Krylov subspace methods.
Numerical experiments illustrate the performance of the proposed method 
in the number of signals detection application, numerical rank estimation of general 
data matrices, and in video foreground detection, an application of PCA.

     \section{Preliminaries}
     We begin by presenting the problem formulation for dimension estimation of the principal subspace.

     \paragraph{Notation:} We use lowercase and
uppercase bold letters, $\mbf{x}$ and $\mbf{A}$ for vectors and matrices, respectively.  
The Gaussian distribution with mean $\mu$ and covariance $\mbf{\bSigma}$
is denoted by $\mathcal{N}(\mu,\mbf{\bSigma})$.
Identity matrix is depicted as $\mbf{I}_p$, where $p$ is the order.
Convergence in distribution is denoted by $\rightarrow_d$.

          \paragraph{Problem Formulation:}\label{sec:model}
     The data observations which form the matrix $\mbf{X}$ are typically modeled as high 
    dimensional random quantities embedded in noise.
     We assume the standard Gaussian random model for the set of $n$ data observations 
     each of  dimension $p$. 
     We denote the $p$-dimensional data
     as $\{\mbf{x}_i\}_{i=1}^n$ described as 
     \begin{equation}
      \mbf{x}_i=\mbf{M}\mbf{s}_i+\sqrt{\sigma} \mbf{n}_i, \: i=1,\ldots,n
     \end{equation}
where $\mbf{M}$ is a $p\times q$ mixing matrix with $q$ independent
columns, $\mbf{s}_i$ are  $q\times 1$
vectors containing
the zero mean relevant data and $\mbf{n}_i$
are
$p$-dimensional
Gaussian (white) noise vectors with parameter $\sigma$
as the unknown
noise variance.  This is a standard assumption made in PCA~\cite{jolliffe2002principal}, probabilistic PCA~\cite{tipping1999probabilistic}, signal detection and subspace tracking~\cite{wax1985detection,xu1994fast2}, and in modern data analysis~\cite{bradde2017pca} and neural networks~\cite{hinton2006reducing} methods.
The true covariance matrix $\bSigma$ associated with the underlying 
      data is then
      assumed to be a low rank matrix of rank $q$, perturbed by noise of variance $\sigma$.  That is,
      \[
       \bSigma=\mbf{B}\mbf{B}^T+\sigma \mbf{I}_p,
      \]
      where $\mbf{B}\in\bR^{p\times q}, q\ll p$ and $span(\mbf{B})$ is the principal subspace. 
      The top $q$ eigenvalues $\lambda_i$ for $i=1,\ldots,q$ of $\bSigma$ will correspond to
      the $q$ dimensional relevant data  and the remaining $p-q$ eigenvalues 
      are related to noise  and are equal to $\sigma$. Hence, the subspace associated with the 
      top $q$  eigenvectors (eigenvalues) forms the principal subspace,
      which is of interest.
      
      The exact covariance matrix of the underlying data will 
      not be available, and hence
      we consider the    sample covariance matrix $\mbf{S}_n
      =\frac{1}{n}\sum_{i=1}^n\mbf{x}_i\mbf{x}_i^T$, using the $n$ (noisy) observations 
      of the data. 
      We wish  to estimate $q$, the dimension of (relevant) data in the observations,
   using the eigenvalues of the sample covariance
matrix $\mbf{S}_n$
denoted by $\ell_1\geq\ell_2\geq\ldots\geq\ell_p$.

   \section{Proposed Method}\label{sec:propose}
   In this section, we first present the proposed method for the
   principal subspace dimension estimation
   and derive it.
  We then discuss the Krylov subspace methods for computing 
  partial eigen-decomposition of matrices, and present the proposed algorithm for
  simultaneous estimating the dimension and computing an approximation
   to the principal subspace.
   
   The proposed  method is  based on model selection technique
and the proposed criterion is the following: 
   \begin{equation}\label{eq:crit}
   \arg\min_k \left[\frac{n}{2\sigma^2}\sum_{i=k+1}^p(\ell_i-\sigma)^2
   -C_n\frac{(p-k)(p-k-1)}{2}\right]     
    \end{equation}
where $\ell_i$, for $i=1,\ldots,p$ are the eigenvalues of the sample 
covariance matrix $\mbf{S}_n=\frac{1}{n}\mbf{XX}^T$, $\sigma$ is the noise variance,
and $C_n$ is a parameter that depends on $n$ (see sec.~\ref{sec:analysis} for details).
Note that the first term in the criterion depends on the sum of bottom $p-k$ eigenvalues of $\mbf{S}_n$, which can be
written as 
\[
 \sum_{i=k+1}^p(\ell_i-\sigma)^2 = \|\mbf{S}_n-\sigma \mbf{I}_p\|_F^2-\sum_{i=1}^k(\ell_i-\sigma)^2.
\]
Thus, the method requires computing only the top $k$ eigenvalues of $\mbf{S}_n$.
We can compute the norm as 
$\|\mbf{S}_n-\sigma \mbf{I}_p\|_F^2=\tfrac{1}{n^2}\|\mbf{X}\|_F^4-\tfrac{2\sigma}{n}\|\mbf{X}\|_F^2+p\sigma^2$.
Therefore, if Krylov subspace method such as the Lanczos algorithm~\cite{Saad-book3} 
is used for computing these eigenvalues, 
   then we do not need  to form $\mbf{S}_n=\frac{1}{n}\mbf{XX}^T$ explicitly.
   
   The Krylov subspace methods will also yield an approximation to the eigenvectors corresponding to the computed 
   eigenvalues. Therefore, we can use the above method as a stopping criterion for
   the Krylov subspace methods,
   and hence, estimate the dimension and approximate  the principal 
   subspace of the covariance matrix,
   simultaneously. We present the resulting algorithm in the latter part of this section.
   First, we derive the above criterion using concepts from random matrix perturbation theory.

   \subsection{Derivation}\label{sec:derive}
   We start the derivation of the proposed selection criterion using the following 
   key concept from random matrix  theory~\cite{muirhead2009aspects}:
  The sample covariance matrix $\mbf{S}_n$ approaches
   the true covariance matrix $\bSigma$ only in the expectation, i.e.,   
   $
    \mathbb{E}[\mbf{S}_n]\rightarrow\bSigma.
   $
More importantly, the sample covariance matrix  $\mbf{S}_n$ is a $\sqrt{n}$ consistent estimator of 
$\bSigma$.
\begin{proposition}\label{prop:1}
 $\mbf{S}_n$ is a $\sqrt{n}$ consistent estimator of 
$\bSigma$. That is, 
\[
 \sqrt{n}\mathtt{vec}(\mbf{S}_n-\bSigma)\rightarrow_d\mathcal{N}(0,\bOmega),
\]
where $\bOmega=(I+P_{\tvec(\mbf{S}_n)})(\bSigma\otimes\bSigma)$ is a $p^2\times p^2$ covariance matrix
with $\otimes$ denoting the Kronecker product and $P_{\tvec(\mbf{S}_n)}$
 the transposition-permutation matrix associated to $\tvec(\mbf{S}_n)$.
\end{proposition}
The  proof of this proposition can be found in most standard multivariate statistical
theory textbooks, e.g.,~\cite{anderson2003introduction,muirhead2009aspects}.

%

Next, we consider the eigen-decomposition of the covariance matrix $\bSigma=\mbf{U}\Lambda\mbf{U}^T$.
Let us write $\mbf{U}=[\mbf{U}_q, \mbf{U}_{p-q}]$, where $\mbf{U}_q$ is a matrix containing the top
$q$ eigenvectors (principal subspace) of $\bSigma$ as columns. Similarly, let us consider the eigen-decomposition 
of the sample covariance matrix $\mbf{S}_n={\mbf{G}}\mbf{L}{\mbf{G}}^T$, with
$\mbf{G}_q$ containing the top $q$ eigenvectors of $\mbf{S}_n$ as columns.
We can then prove the consistency of $\mbf{G}_q$ using the random matrix perturbation approach
on $\mbf{S}_n$.

\begin{proposition}\label{prop:2}
 Let $q$ be the numerical rank of $\bSigma$ and assume that the smallest  eigenvalue corresponding to the data
 is well above zero, i.e.,  that 
 $\lambda_q>\veps>0$ for a small $\veps$. Then as $n\rightarrow\infty$,
 \[
  \mbf{G}_q\rightarrow_d\mbf{U}_q.
 \]
\end{proposition}
A version of the proof of this proposition is given in the supplementary, which was first derived in~\cite{anderson1963asymptotic}.
We then have the following result (proof in the supplementary).

\begin{corollary}
 The orthogonal projector onto the space spanned by the eigenvectors corresponding to the noise related
 eigenvalues satisfies
 \[
  {\mbf{Q}}_G = \mbf{G}_{p-q}\mbf{G}_{p-q}^T=\mbf{U}_{p-q}\mbf{U}_{p-q}^T+O_p\left(\frac{1}{\sqrt{n}}\right).
 \]

\end{corollary}
We next have the following result that gives the asymptotic behavior of
the bottom $p-q$ eigenvalues of 
$\mbf{S}_n$.

 \begin{proposition}\label{prop:3}
 The asymptotic distribution of $\sqrt{n}\tvec({\mbf{Q}}_G(\mbf{S}_n-\sigma I){\mbf{Q}}_G)$ is given by
 \[
  \sqrt{n}\tvec({\mbf{Q}}_G(\mbf{S}_n-\sigma \mbf{I}_p){\mbf{Q}}_G)\rightarrow
  \mathcal{N}(0,\hat\bOmega),
 \]
where $\hat\bOmega = (\mbf{Q}_U \otimes\mbf{Q}_U)\bOmega  (\mbf{Q}_U \otimes\mbf{Q}_U)$,
where $\mbf{Q}_U=\mbf{U}_{p-q}\mbf{U}_{p-q}^T$ and $\bOmega$ is as Proposition~\ref{prop:1}.
\end{proposition}
We defer the proof to the supplementary.
This leads to the following result.
\begin{lemma}\label{lemm:1}
 Let $\mathcal{L}$ be defined as 
 \[
  \mathcal{L}=\frac{n}{2\sigma^2}\sum_{i=q+1}^p(\ell_i-\sigma)^2,
 \]
where $\ell_i$ are the eigenvalues of $\mbf{S}_n$ and $\sigma$ is the noise variance. 
Then  $\mathcal{L}$ follows asymptotically
a $\chi^2$ 
chi-square distribution with $\eta=\frac{1}{2}(p-q)(p-q-1)$ degrees of freedom. 
\end{lemma}
\begin{proof}
Suppose $L_{p-q}$ is a diagonal matrix with the bottom $p-q$ eigenvalues of $ \mbf{S}_n-\sigma \mbf{I}_p$
as entries, then we have
\begin{eqnarray*}
 n\sum_{i=q+1}^p(\ell_i-\sigma)^2 &=& n \tr(L_{p-q}^2)\\
  &=& n \tr({\mbf{Q}}_G (\mbf{S}_n-\sigma \mbf{I}_p)^2{\mbf{Q}}_G) \\
   &=& \|\sqrt{n} ({\mbf{Q}}_G (\mbf{S}_n-\sigma \mbf{I}_p){\mbf{Q}}_G)\|_F^2\\
   &=& \|\sqrt{n} \tvec({\mbf{Q}}_G (\mbf{S}_n-\sigma \mbf{I}_p){\mbf{Q}}_G)\|_2^2.
\end{eqnarray*}
From Proposition~\ref{prop:3}, the above sum follows asymptotically 
a  $\eta=\tfrac{1}{2}(p-q)(p-q-1)$ weighted $\chi_1^2$ distribution~\cite{anderson2003introduction}, 
where the $\eta$ weights correspond to the first $\eta$ 
eigenvalues 
of $\hat\bOmega = (\mbf{Q}_U \otimes\mbf{Q}_U)\bOmega  (\mbf{Q}_U \otimes\mbf{Q}_U)$.
 Note that $\eta$ is the degree of freedom in ${\mbf{Q}}_G (\mbf{S}_n-\sigma \mbf{I}_p){\mbf{Q}}_G$.

Given the eigenpairs of $\bSigma$ to be $(\lambda_i,\mbf{u}_i)\:, i=1,\ldots,p$, 
the eigenpairs of $\bOmega$ will be
$(\lambda_i*\lambda_j,\mbf{u}_i\otimes \mbf{u}_j),\: i,j=1,\ldots,p$ from the property of Kronecker products,
see~\cite[Thm. 4.2.12]{horn1990matrix}. $\mbf{Q}_U$ is a projector onto the span of eigenvectors 
corresponding to the bottom $p-q$ eigenvalues of $\bSigma$, which are all equal to $\sigma$.
Hence, the top $\eta$ eigenvalues of  $\hat\bOmega$ will be all equal to $\sigma^2$, since
$(\mbf{Q}_U \otimes\mbf{Q}_U)$ is a projector onto space spanned by the eigenvectors corresponding to 
the bottom $(p-q)^2$ eigenvalues of $\bOmega$. Hence, the weights of the weighted $\chi^2$ are all equal
to $\sigma^2$. Thus, by reweighting the above sum, $ \mathcal{L}$ will have asymptotically 
$\chi_\eta^2$ distribution\footnote{Anderson made a similar observation (of asymptotically $\chi_\eta^2$ distribution)
in~\cite{anderson1963asymptotic} for a given eigenvalue $\lambda_k$ of $\bSigma$ 
with multiplicity $q_k$ and the sum of eigenvalues of $(\mbf{S}_n-\lambda_k\mbf{I})$.
In our case, $\lambda_k=\sigma$ with multiplicity $q_k=p-q$.}.
\end{proof}

Therefore, the above
$ \mathcal{L}(\mbf{S}_n,q)$
can be used in model selection criterion for estimating $q$, the dimension of the principal subspace.

\begin{theorem}\label{theo:criteria}
The following criterion  yields an estimation for the dimension $q$
of the principal subspace of the covariance matrix $\bSigma$:
\begin{equation}
   q=\arg\min_k \left[\frac{n}{2\sigma^2}\sum_{i=k+1}^p(\ell_i-\sigma)^2
   -C_n\tfrac{(p-k)(p-k-1)}{2}\right],     
    \end{equation}
where $\ell_i$, for $i=1,\ldots,p$ are eigenvalues of the sample 
covariance matrix $\mbf{S}_n=\frac{1}{n}\mbf{XX}^T$, $\sigma$ is the noise variance
and $C_n$ is a parameter that depends on $n$.
\end{theorem}

Proof of the theorem is given in the supplementary.
We also give a simulation result which shows that Lemma~\ref{lemm:1} and Theorem~\ref{theo:criteria} 
hold true in practice.

\subsection{Krylov subspace methods}
Krylov 
subspace methods
are popularly used to compute the partial spectrum (top $k$ eigenvalues and eigenvectors) of 
matrices~\cite{Saad-book3}. Recent results~\cite{musco2015randomized} have shown that 
these methods return high quality principal components and give nearly optimal
PCA for any matrix. 
The proposed dimension estimation criterion  can be used as a stopping criterion for such 
Krylov subspace approximation
of the principal subspace of covariance matrices.

For a symmetric matrix $\mbf{A}$, the Krylov 
subspace is defined as $\mathcal{K}^m(\mbf{A},\mbf{v})=span\{\mbf{v},\mbf{Av},\ldots,\mbf{A}^{m-1}\mbf{v}\}$,
where $\mbf{v}$ is a random vector of unit norm, $\|\mbf{v}\|=1$, $\mbf{v}\nsubseteq null(\mbf{A})$   and $m$ is a scalar. 
The Lanczos algorithm builds an orthonormal basis for this Krylov subspace~\cite{Saad-book3}.
We can also define a block Krylov subspace as: 
$\mathbb{K}^m(\mbf{A},\mbf{V})=span\{\mbf{V},\mbf{AV},\ldots,\mbf{A}^{m-1}\mbf{V}\}$,
where $\mbf{V}\in\bR^{p\times k}$ is a random matrix such that $\mbf{V}\nsubseteq null(\mbf{A})$,
see~\cite{musco2015randomized} for recent theoretical results for randomized 
block Krylov subspace methods.
We can compute approximate  eigenvalues and eigenvectors of $\mbf{A}$, say $\{\theta_i,\mbf{y}_i\}_{i=1}^k$ for some $k$,
 using the Krylov subspace methods. We have the following result 
from eqn.~3 and Theorem~1 in~\cite{musco2015randomized}:

\begin{algorithm}[tb!]
\caption{Proposed Algorithm}
\label{alg:algo1}
\begin{algorithmic}
   \STATE {\bfseries Input:} Data matrix $\mbf{X}\in\bR^{p\times n}$,  noise variance $\sigma$, parameter $C_n$, and a error tolerance $\eps$.
 \STATE {\bfseries Output:} Dimension $q$ and an approximation to the principal subspace $\mbf{Y}_q$.
 \STATE Set $IC$=$zeros(p,1)$, $\mbf{Q}=[\:], k=1$, $m=\frac{\log(p)}{\sqrt{\eps}}$, $\Phi=\tfrac{1}{n^2}\|\mbf{X}\|_F^4-\tfrac{2\sigma}{n}\|\mbf{X}\|_F^2+p\sigma^2$.
\FOR{$k=1$ to $p$}
\STATE{\bfseries 1.} Generate a random  vector $\mbf{v}_k$  with $\|\mbf{v}_k\|_2=1$.
\STATE {\bfseries 2.} $\mbf{K}=\frac{1}{n}[\mbf{Xv}_k,(\mbf{XX}^T)\mbf{Xv}_k,\ldots,(\mbf{XX}^T)^{m-1}\mbf{Xv}_k]$
\STATE {\bfseries 3.}  $\mbf{Q}=orth([ \mbf{Q}, \mbf{K}])$, $\mbf{Q}=\mbf{Q}(:,1:k)$.
\STATE {\bfseries 4.} $\mbf{T}=\frac{1}{n}\mbf{Q}^T\mbf{X}\mbf{X}^T\mbf{Q}$.
\STATE {\bfseries 5.} $[\mbf{V},\Theta]=\texttt{eig}(\mbf{T})$.
\STATE {\bfseries 6.} $IC(k)=n(\Phi-\sum_{i=1}^k(\theta_i-\sigma)^2)
   -C_n\tfrac{(p-k)(p-k-1)}{2}$
   \IF{($k>1$ \&\& $IC(k)> IC(k-1)$)} 
   \STATE break;
   \ENDIF
\ENDFOR
\STATE $q=k-1$.
Output $q$ and $\mbf{Y}=\mbf{QV}$.
\end{algorithmic}
\end{algorithm}

\begin{lemma}\label{lemm:2}
 Consider a symmetric PSD matrix $\mbf{A}\in\bR^{p\times p}$ with eigenvalues $\ell_i,i=1,\ldots,p$.
Let $\{\theta_i,\mbf{y}_i\}_{i=1}^k$ be the $k$ eigenpair computed using  $m$ steps of 
block Krylov subspace method (using the orthonormal basis of $\mathbb{K}^m(\mbf{A},\mbf{V})$ for 
$\mbf{V}\in\bR^{p\times k}$). If $m=\frac{\log(p)}{\sqrt{\eps}}$ for some $0<\eps<1$, then we have
\[
 |\theta_i-\ell_i|\leq\eps\ell_{k+1},\: i=1,\ldots, k.
\]
Moreover, suppose $\mbf{Y}_k$ is a matrix containing the eigenvectors $\{\mbf{y}_i\}_{i=1}^k$ computed by the 
 Krylov subspace method as columns, then we have for $\xi\in\{2,F\}$
 \[
  \|\mbf{A}-\mbf{Y}_k\mbf{Y}_k^T\mbf{A}\|_\xi\leq (1+\eps)\|\mbf{A}-\mbf{A}_k\|_\xi,
 \]
where $\mbf{A}_k$ is the best rank $k$ approximation of $\mbf{A}$ obtained using its exact eigen-decomposition.
\end{lemma}
Therefore,  the Krylov subspace method will return a high quality principal components of $\mbf{S}_n$
and near optimal $(1+\eps)$ PCA. In addition, 
the eigenvalues $\theta_i$'s computed are very close to the actual eigenvalues $\ell_i$s of the sample
 covariance matrix (within a multiplicative factor).
 The error $\eps$ in the above analysis is related to the gap in the spectrum, i.e., 
we can replace $\eps$ by
$\frac{\ell_{k}}{\ell_{k+1}}-1$, see~\cite[\S 7]{musco2015randomized}. 
For $k>q$, the error term $\eps\ell_{k+1}$ is related to the
noise related eigenvalues and we have $\eps\ell_{k+1}=O(\frac{1}{\sqrt{n}})$
from the analysis in section~\ref{sec:derive} and \cite{anderson1963asymptotic}.
Asymptotically, this term goes to zero.
Thus, $\theta_i$'s have the same statistical properties of $\ell_i$'s, and are good approximation to them.
Since $\ell_i$'s are asymptotically equivalent to $\lambda_i$'s, $\theta_i$'s are good estimates of 
$\lambda_i$'s.

\paragraph{Proposed Algorithm:}
Algorithm~\ref{alg:algo1} presents the proposed algorithm for simultaneously estimating the dimension  and computing the principal subspace of the covariance matrix.
In step 2, note that only matrix-vector products with the data $\mbf{X}$ and its transpose are needed to 
form the Krylov matrix $\mbf{K}$. In step  3, since $\mbf{Q}$ is already orthonormal from the previous iteration,
the new vectors in $\mbf{K}$
can be quickly orthonormalized wrt. $\mbf{Q}$.
We can also replace steps 2-5, by a version of the Lanczos algorithm~\cite{Saad-book3}, 
which updates the previous subspace $\mbf{Q}$ and the tridiagonal matrix $\mbf{T}$. 

{\bf \emph{Cost:}} If $q$ is the exact dimension, the computational cost of the algorithm will be
$O(\nnz(\mbf{X})qm+p(qm)^2)$, where $\nnz(\mbf{X})$ is the number of nonzeros in $\mbf{X}$. 
Since both $q\ll p$ and $m=\frac{\log(p)}{\sqrt{\eps}}$ are small, the algorithm is quite inexpensive, more so 
if  data $\mbf{X}$ is sparse.

{\bf\emph{Choosing $\sigma$:}}
 In our Algorithm,
 we need to choose  the noise level $\sigma$, 
  when it is unknown. 
 In many applications, e.g., in signal processing, 
 typically an estimate of noise level is known.
 In  low rank approximation problems, the maximum approximation error tolerance
  acceptable might be known.
 Otherwise, for signal processing applications, $\sigma$ can be 
 determined using the thresholding method proposed in~\cite{kritchman2009non}.
 For data related applications, article~\cite{ubaru2017fast} discusses an inexpensive method 
 to estimate $\sigma$ using the spectral density plot of the data matrix.
 For further details, see~\cite{ubaru2016fast,ubaru2017fast}.

\section{Analysis}\label{sec:analysis}
In this section, we first show that the proposed method yields
a  strong consistent estimator for $q$, the exact dimension.
We then analyze 
the conditions for correct estimation for finite $n$ data observations.

\subsection{Strong consistency}
\begin{theorem}\label{theo:1}
 The criterion defined by 
 \begin{equation}
  IC(k)=\frac{n}{2\sigma^2}\sum_{i=k+1}^p(\ell_i-\sigma)^2
   -C_n\frac{(p-k)(p-k-1)}{2}
 \end{equation}
can be used to obtain a strong consistent estimator for $q$, the exact dimension of the
principal subspace, i.e.,
$\lim_{n\rightarrow\infty}\hat{k}=q$, where $\hat{k}= \arg\min_k IC(k)$,
with value of $C_n$ such that 
\[\lim_{n\rightarrow\infty}\frac{C_n}{n}=0
\text{ and } \lim_{n\rightarrow\infty}\frac{C_n}{\log\log n}=\infty.
\]
\end{theorem}
Proof of this theorem is given in the supplementary.
For the right choice of $C_n$, the proposed estimator is strongly consistent.
Next, we consider the eigenvalues computed using the Krylov subspace method
in our criterion. 

\begin{corollary}\label{corr:2}
 For the choice of $C_n$ in Theorem~\ref{theo:1}, the criterion~\ref{eq:crit} 
 is strongly consistent for the 
 eigenvalues computed using the Krylov subspace method in Algorithm~\ref{alg:algo1}
 if we set the parameter $\sigma=(1-\eps)\sigma_{true}$ in the algorithm, where
 $\sigma_{true}$ is the true noise variance of the data.
\end{corollary}
The proof can be found in the supplementary.
Next, we analyze the performance of the proposed method for finite sample size and obtain the
conditions for  correct detection.
  
\subsection{Performance Analysis}
The consistency analysis above considered the asymptotic case when $n\rightarrow\infty$, and the
 law of iterated logarithm~\cite{muirhead2009aspects} is used to derive the results. 
 Here, we  analyze the performance of the
 proposed method for finite sample size
 (general $n$), and obtain the conditions when the method either \emph{underestimates} or 
 \emph{overestimates} the dimension. 
 
      \begin{figure*}[tb!]
    \begin{center}
\includegraphics[width=0.32\textwidth]{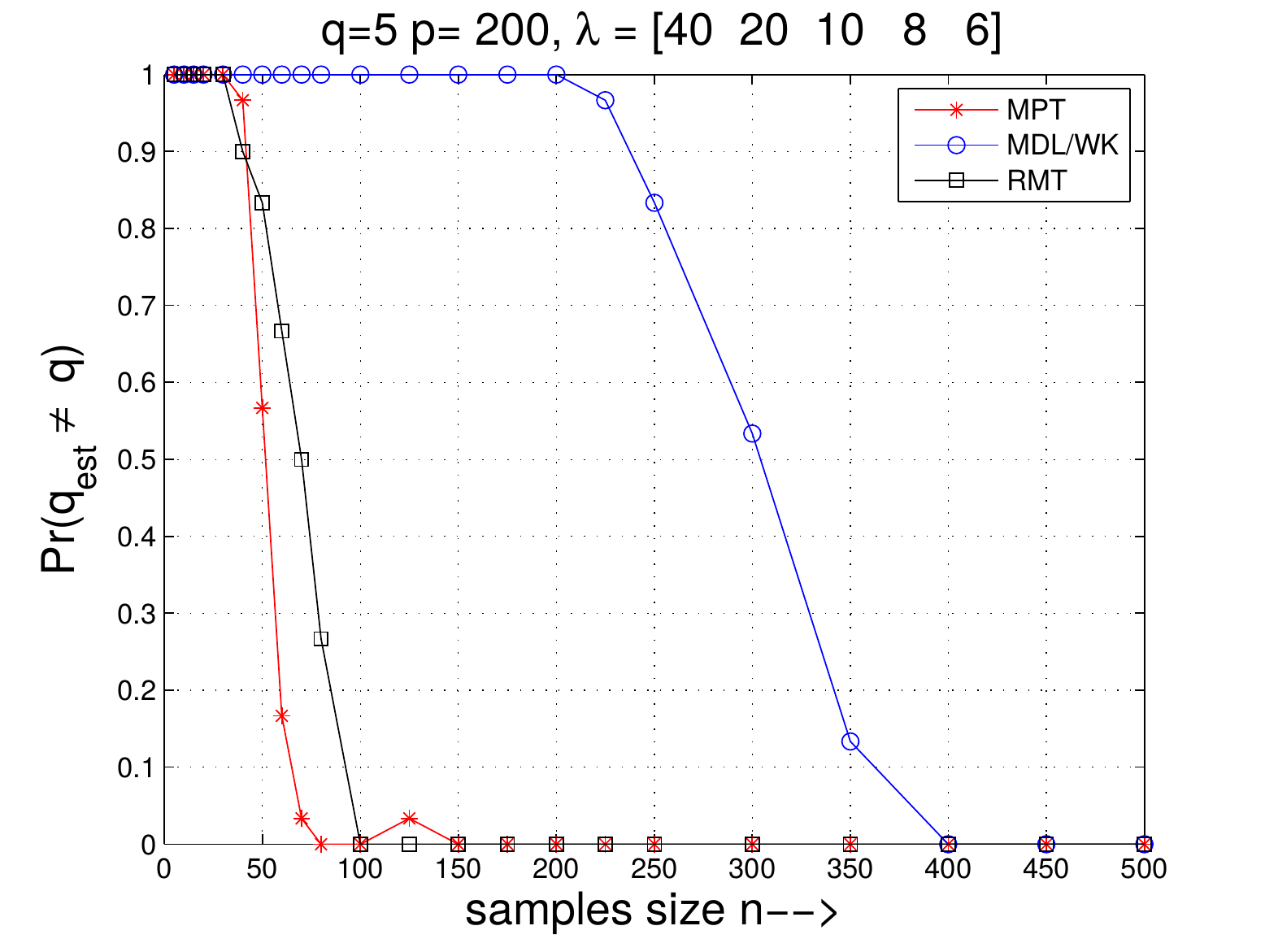}
\includegraphics[width=0.32\textwidth]{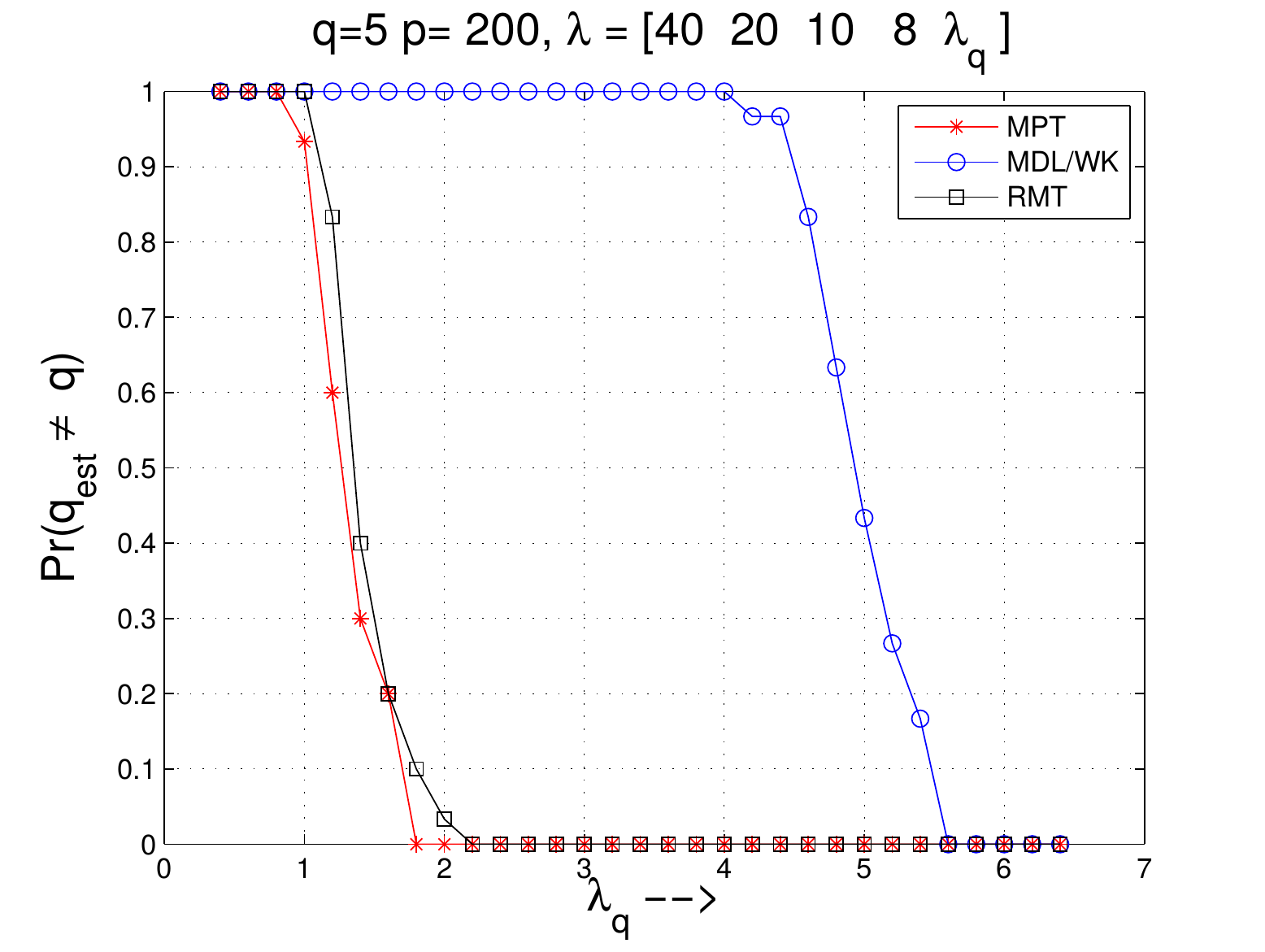}
\includegraphics[width=0.32\textwidth]{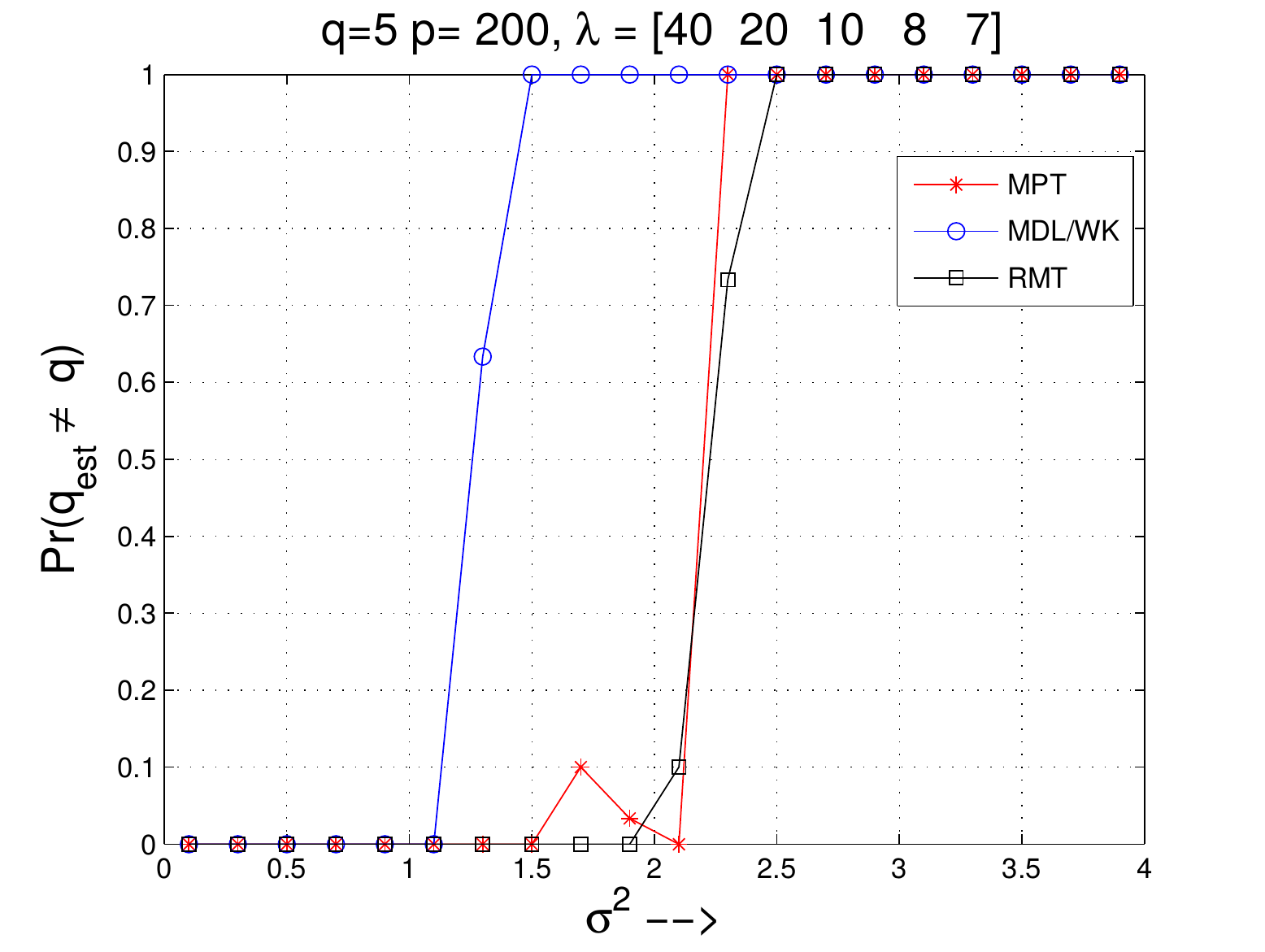}
\vskip -0.1in
 \caption{Signal detection: Comparison between the proposed method MPT, RMT and MDL 
 as a function of: (left) the number of samples $n$, (middle)
 signal strength 
 ($\lambda_q$ eigenvalue), and (right) the noise level $\sigma$.}\label{fig:1}
 \end{center}
  \vskip -0.2in
\end{figure*} 

 The notorious scenario for wrong detection is when the dimension is off by exactly one $(q\pm 1)$, which we 
 analyze here (important in signal detection applications). The analysis trivially generalizes to  other cases.
 First, let us consider underestimation by one, and consider the following difference:
 \begin{eqnarray*}
  \Delta_1 & = & IC(q-1) -IC(q)\\
  &=& \frac{n}{2\sigma^2}(\ell_q-\sigma)^2 - C_n(p-q).
 \end{eqnarray*}
 Note that we will not have underestimation when $\Delta_1>0$, i.e., when 
 \[
  \ell_q >\sigma\left(\sqrt{\tfrac{2C_n}{n}(p-q)}+1\right).
 \]
So,  we need 
 the magnitude of $\ell_q$ (related to relevant data or the signal strength) to be 
large enough in order  to avoid underestimate the dimension.
That is, we need a reasonable gap between relevant eigenvalues and the noise related eigenvalues in the spectrum.
For the asymptotic case ($n\rightarrow\infty$), 
we know that the RHS term with $C_n$ goes to zero and, hence we will not have any underestimation 
of  dimension as long as the signal strength is 
more than the noise variance.

 Next, let us consider overestimation of the dimension by one, and  the following difference:
  \begin{eqnarray*}
  \Delta_2 & = & IC(q+1) -IC(q)\\
  &=&  C_n(p-q-1)- \frac{n}{2\sigma^2}(\ell_{q+1}-\sigma)^2.
 \end{eqnarray*}
 Again, we will not overestimate if $\Delta_2>0$, i.e., when 
 \[
 \frac{\ell_{q+1}}{\sigma}<\sqrt{\tfrac{C_n}{n}(p-q-1)}+1.
 \]
 We know that $\ell_{q+1}$ corresponds to the largest noise related eigenvalue of the covariance matrix.
 For the asymptotic case ($n\rightarrow\infty$), we know $\ell_{q+1}\rightarrow\sigma$, hence
 the equation holds.
 For finite $n$, we must choose the noise parameter $\sigma$ close to the true noise level 
 (reflected in $\ell_{q+1}$) in order to avoid overestimation.
 Assuming the noise variance $\sigma$ is known, for finite $n$, 
 when the ratio of $p/n$ or $n/p$ is not too large,
 we can derive bounds on the parameter $C_n$ in our method to avoid overestimation, 
 using the  random matrix theory 
 results in~\cite{johansson2000shape,johnstone2001distribution}.
 
 The largest eigenvalue of the sample covariance matrix (Wishart matrix) of
pure noise vectors with Gaussian distribution  follows the Tracy-Widom 
distribution~\cite{johansson2000shape,johnstone2001distribution}.
Then, for finite $p,n$ as long as $\min\{p,n\}\gg1$ and the ratio of $p/n$ or $n/p$ is not too large, 
the  largest eigenvalue due to noise will be  approximately $\sigma(1+\sqrt{p/n})^2$, 
see~\cite{kritchman2009non} for details.
Hence, for finite  but large values of $p,n$, we have 
\[
 \ell_{q+1}\approx \sigma\left(1+\sqrt{\frac{p}{n}}\right)^2.
\]
Substituting in the condition above for overestimation, we get the following bound for the parameter $C_n$
for exact detection for finite but large values of $p,n$:
\[
 C_n>\frac{(p+2\sqrt{np})^2}{n(p-q-1)}.
\]
When the ratio of $p/n$ or $n/p$ is not too large, the RHS is fairly small.
The above analysis provides us the conditions on the  relevant eigenvalue $\ell_q$,
noise level and the parameter $C_n$
in order to avoid incorrect estimation of the dimension $q$ using the proposed method.

When we consider the eigenvalues obtained by the Krylov subspace method 
in the criterion, we will have an additional term that depends on $\eps$
in the denominators of the above conditions. That is, we have approximately 
the following conditions for exact dimension detection:
\[
\ell_q >\frac{\sigma}{(1-\eps)}\left(\sqrt{\tfrac{2C_n}{n}(p-q)}+1\right)\text{ and } \]
\[\frac{\ell_{q+1}}{\sigma}<\frac{1}{(1-\eps)}\sqrt{\tfrac{C_n}{n}(p-q-1)}+1.
\]
For small $\eps$, we end up with similar conditions on  $\ell_q$, noise level and $C_n$ as above.

\section{Numerical experiments}\label{sec:expt}
In this section, we present some numerical experimental results to illustrate the 
performance of the proposed method, and compare it to few other popular methods.
First, we consider examples for the number of signals detection application in signal and array processing.
We then consider few large data matrices and a PCA application to illustrate the method's performance.

\begin{table*}
\caption{Performance of the Krylov Subspace method, Algorithm~\ref{alg:algo1} with $m=10$.}
\label{tab:1}
\begin{center}
\begin{tabular}{|l|c|c|c|c|c|c|c|}
\hline
Dataset&$p$ & Actual $q$&$\lambda_q$ &$\sigma$ &Estimated $\tilde{q}$& 
$\|\mbf{A}-\mbf{Y}_{\tilde{q}}\mbf{Y}_{\tilde{q}}^T\mbf{A}\|_F$ & Runtime\\
\hline
sprand& 5000 &50&5&1&50&134.47&6.1 secs\\
      & 5000 &100&2&0.5&100&159.23&22.8 secs\\
      & 10000 &100&2&0.5&100&162.52&72.5 secs\\
      & 40000 &100&2&0.5&100&183.74&101.6 secs\\
       & 100000 &100&2&0.5&100&210.86&192.1 secs\\
  \hline
Harvard& 500&63&2.6&1&69&36.14&0.24 secs\\
lpiceria3d& 3576&108&5&1&104&140.52&0.68 secs\\
EVA& 8497&165&5.2&1&172&81.47&2.90 secs\\
lpstocfor3& 16675&981&23.7&3&981&3.05e4&2.29 secs\\
as-22july& 22963&241&54.6&10&237&311.23&137.4 secs\\
internet& 124651&--&--&1&351&7.49e3&797.8 secs\\
\hline
\end{tabular}
\end{center}
  \vskip -0.2in
\end{table*}

\subsection{Number of signals detection}
In the first set of experiments, we consider the signal detection problem to illustrate the accuracy of the proposed method for dimension estimation (exact detection is desired in this application). The results and observations from these experiments are applicable for general data too, see supplementary.
We consider $p$ dimensional signals $\mbf{x}_i$'s that are 
corrupted by white noise with $\mathcal{N}(0,\sigma\mbf{I})$,
 variance $\sigma$. There are three parameters in this model, 
 namely the number of samples $n$, 
 the  signal strength or the magnitude of the eigenvalue
 $\lambda_q$, and the noise level $\sigma$. We compare the performances of the proposed method,
  the MDL (Minimum Description Length) method proposed in~\cite{wax1985detection}, 
and the `state of the  art' hypothesis testing method proposed 
 in~\cite{kritchman2009non} based on random matrix theory (RMT)
 for signal detection 
 as a function of these three parameters. 
 In all experiments, we set $C_n=\log n$ to ensure that the asymptotic properties  and the finite sample lower bound  on $C_n$  above 
 hold.
 
 Figure~\ref{fig:1} presents three results for the three methods,
 the proposed matrix perturbation theory (MPT) based method, the MDL method and the 
  random matrix theory (RMT) based hypothesis testing method.
 For a chosen signal dimension $p$ (reported in the plot), 
 we generate the signals and the sample covariance matrix based on the 
 considered signal eigenvalues $\lambda$ (listed in the plot).
 We then add noise covariance matrix corresponding to the 
 noise level $\sigma$ considered. 
 We plot the probability of the estimated rank $q_{est}$ being not equal to the 
 actual rank $q$, i.e., $Pr(q_{est}\neq q)$ over 100 trials. 
 In the first plot of Fig.~\ref{fig:1}, we plot $Pr(q_{est}\neq q)$
 as a function of the number of samples $n$. We consider small signal dimension  $p=200$
  (note that MDL and RMT require complete eigen-decomposition), 
 the actual rank $q=5$ and 
 the noise level $\sigma=1.1$. The eigenvalues corresponding to the signals are given in the plot. 
 We note that MDL requires $n\geq p$ to yield exact rank, where as 
 the proposed method MPT yields exact rank for much smaller sample size, and 
 performs even slightly better than 
 the state of the art method RMT which requires all the eigenvalues of the sample covariance matrix. 
 
 In the second (middle) plot, we compare the performances wrt. the signal strength, 
 i.e., the magnitude of the $q$th eigenvalue $\lambda_q$ of the covariance matrix. 
 Again the 
 signal dimension is $p=200$, the actual rank $q=5$ and 
 the noise level $\sigma=1.1$. The number of samples is $n=400$. 
 We note that, the proposed method again outperforms MDL and yields more accurate results
 for much lower signal strength. In the last plot, we compare the 
 performances with respect to the noise level $\sigma$. Here too, the 
 signal dimension is $p=200$, the actual rank $q=5$ and 
 the  number of samples is $n=400$. The signal eigenvalues
 are given in the plot and the signal strength
 $\lambda_q=6$. The proposed method MPT performs better 
 than MDL wrt. the noise level too and performs was well as RMT.
 RMT requires parameters, 
such as confidence level $\alpha$ to be selected. 
More importantly, both MDL and RMT require  computing all the
eigenvalues of the sample covariance
matrix. 
Results for our algorithm~\ref{alg:algo1}
are reported in the supplementary.
\begin{figure*}[t!] 
\begin{center}
\begin{tabular}{cc|cc}
\includegraphics[width=0.2\textwidth,trim={0.5cm 0.5cm 0.5cm 0.5cm}]{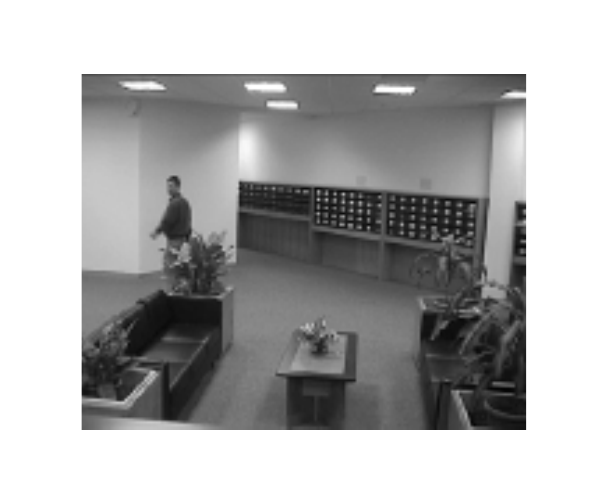} &
\includegraphics[width=0.2\textwidth,trim={0.5cm 0.5cm 0.5cm 0.5cm}]{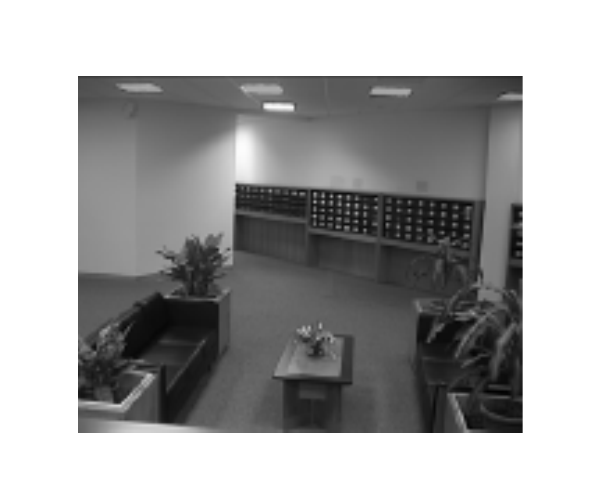} &
\includegraphics[width=0.24\textwidth,trim={0.5cm 0.5cm 0.5cm 0.5cm}]{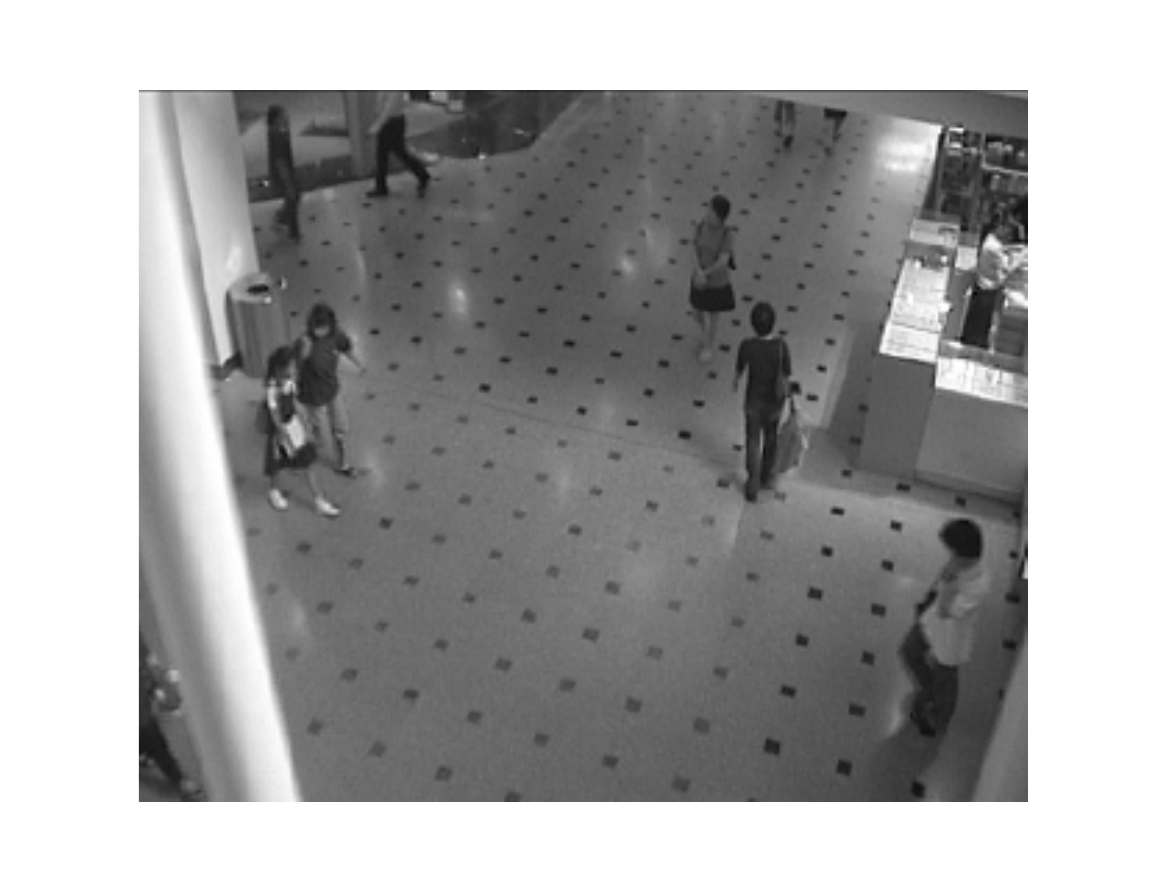} &
\includegraphics[width=0.24\textwidth,trim={0.5cm 0.5cm 0.5cm 0.5cm}]{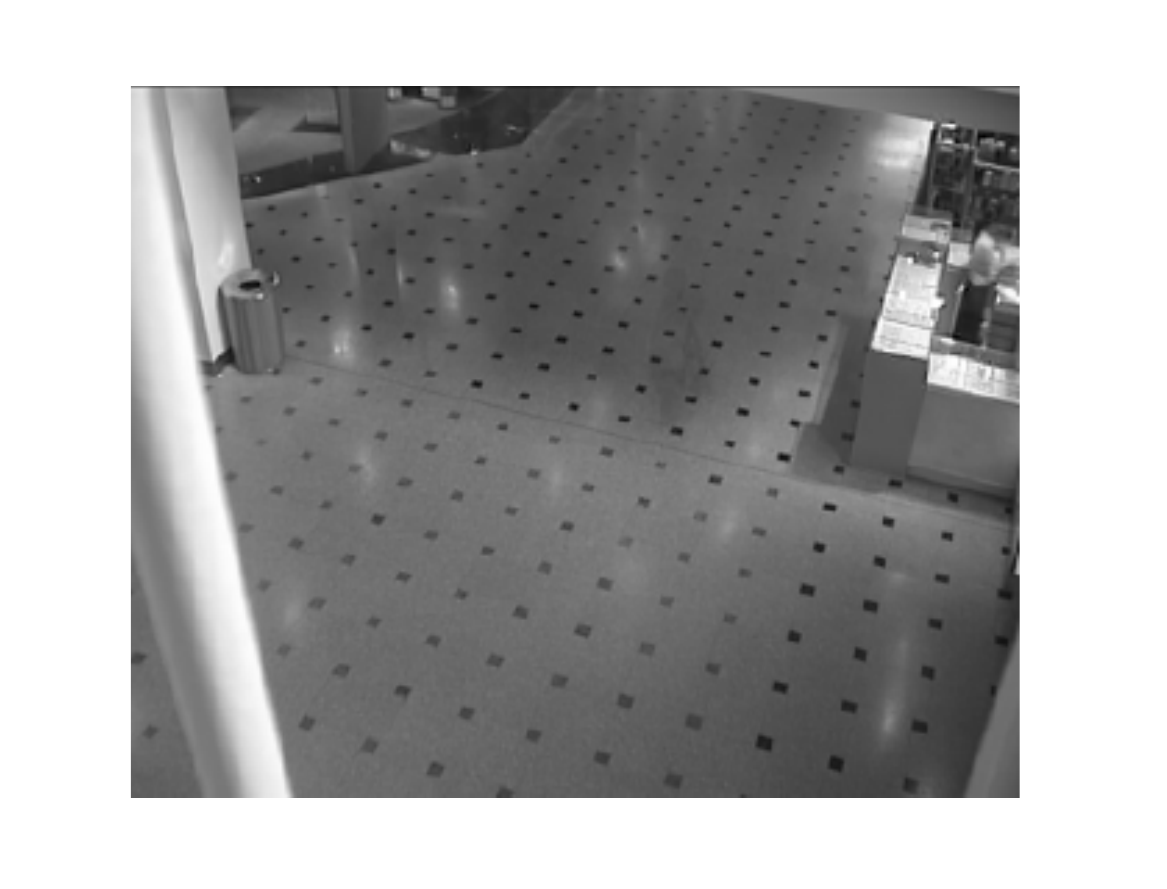} \\
\includegraphics[width=0.2\textwidth,trim={0.5cm 0.5cm 0.5cm 0.5cm}]{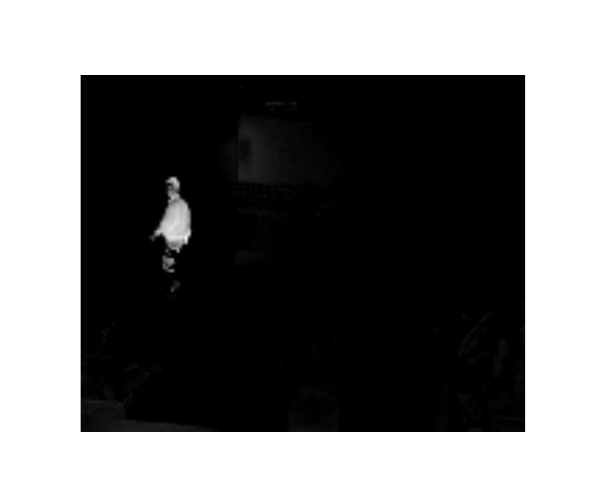} &
\includegraphics[width=0.2\textwidth,trim={0.5cm 0.5cm 0.5cm 0.5cm}]{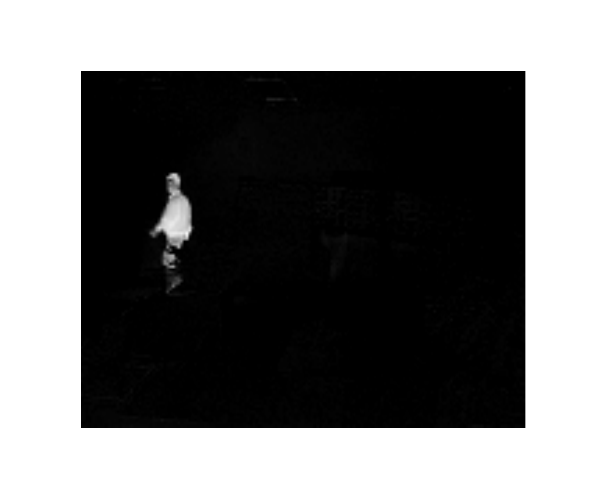} &
\includegraphics[width=0.24\textwidth,trim={0.5cm 0.5cm 0.5cm 0.5cm}]{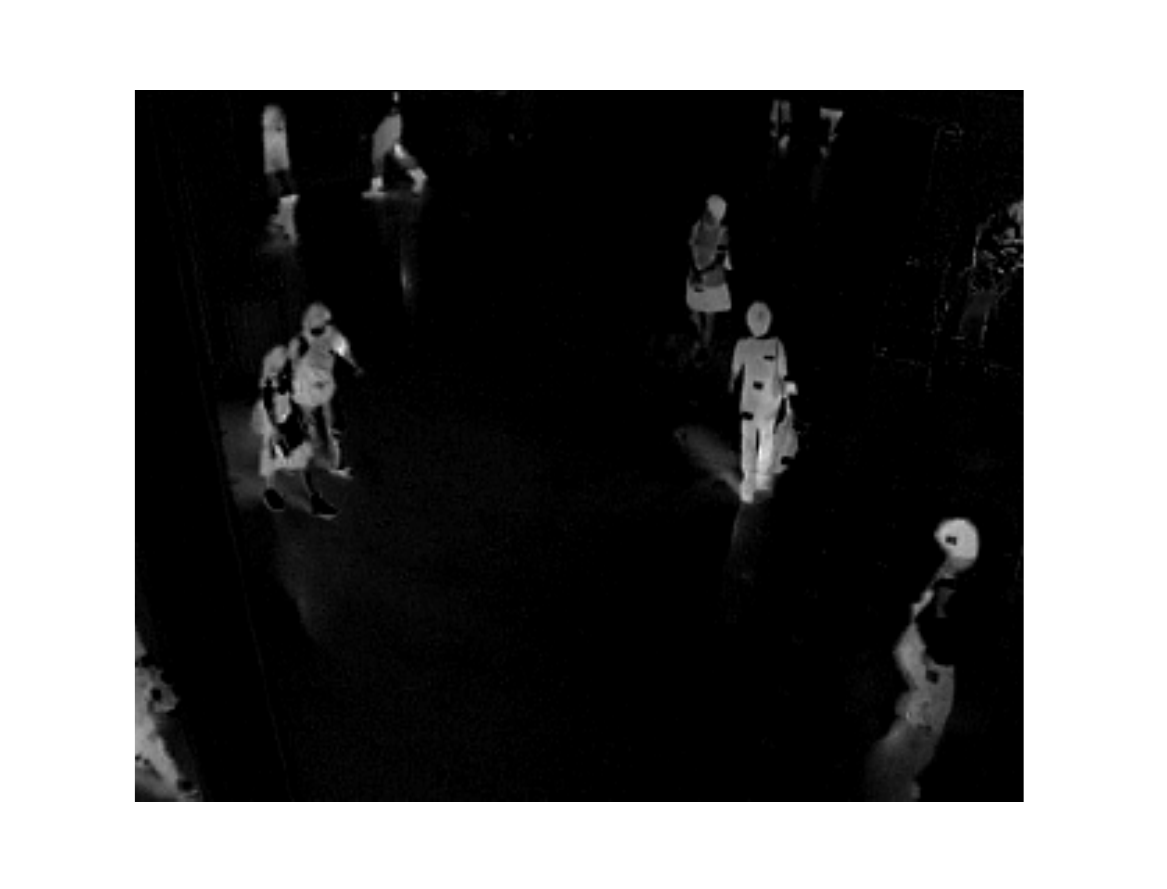} &
\includegraphics[width=0.24\textwidth,trim={0.5cm 0.5cm 0.5cm 0.5cm}]{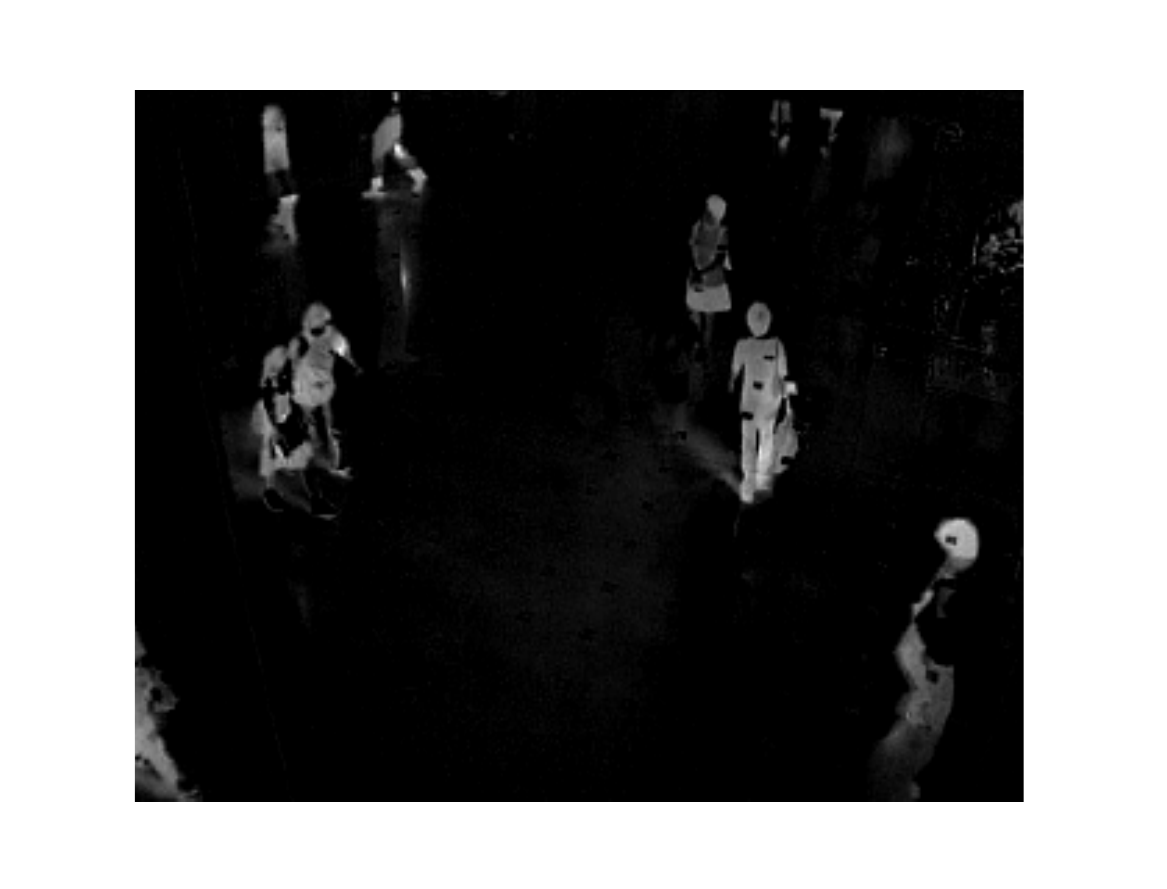} \\
\end{tabular}
\caption{\emph{Background subtraction:} for two sample images from two video datasets.
Low rank approximation (mean added) and foreground detection 
with eigenvectors from  proposed method and  exact 
eigenvectors.}
\label{fig:video1}
\end{center}
\vskip -0.2in
\end{figure*}

\subsection{Data matrices}
Next, we illustrate the performance 
of the proposed method for numerical rank estimation of data matrices.
We consider general data matrices that have low numerical rank from 
publicly available database, 
SuitSparse~\cite{davis2011university}, and a few synthetic sparse random matrices.
For these matrices,
the Gaussian type distribution assumptions
for the data and noise may not  hold. 
 We report additional comparative results in the supplementary.

Table~\ref{tab:1} presents the performance of the Krylov Subspace method, i.e., 
Algorithm~\ref{alg:algo1} for dimension estimation and approximation of the principal
subspace. The synthetic sparse random matrices are of the form 
$\mbf{X}=\mbf{B}\Lambda\mbf{B}^T+\mbf{N}$, where $\mbf{B}$ is a sparse (relevant) data matrix (unit column norm) of size 
$p\times q$ (sparsity $\text{nnz}(\mbf{B})/pq=[0.05,0.1]$), $\Lambda$ is a diagonal matrix 
with the smallest diagonal entry equal to $\lambda_q$ listed in the 4th column.
$\mbf{N}$ is a Gaussian sparse random matrix with $\sigma$ listed in fifth column.
The number of Lanczos steps per iteration (for each $k$) is $m=10$.
The exact dimension $q$ and the estimated dimension $\tilde{q}$ are reported (dimension estimation),
along with the Frobenius norm error$\|\mbf{A}-\mbf{Y}_{\tilde{q}}\mbf{Y}_{\tilde{q}}^T\mbf{A}\|_F$, evaluating
the quality of approximation to the principal subspace. 
The runtime of the algorithm is also reported (computed using \texttt{cputime}
function on an Intel i-5 3.4GHz machine).
For the synthetic examples, we vary the  parameters: size $p$, rank $q$, 
data strength $\lambda_q$ and noise level $\sigma$, and report the results. 
We also consider a few sparse data matrices (also see supplementary). 
We report matrices that have  smaller numerical rank ($q\ll \min(n,p)$)
and a reasonable gap in the spectrum.
 The Krylov subspace  algorithm works well only when there is a 
spectral gap. Otherwise, the interior eigenvalues do not converge. 
For large matrix 'internet', we do not know the exact rank (cannot compute complete decomposition).
We observe that the algorithm performs reasonably well for these  matrices.
The method is also quite inexpensive, particularly for large sparse data matrices.

\subsection{Video Foreground Detection}
In the last experiment, we consider an application of PCA, that of 
background subtraction  in surveillance videos.
Here, PCA is used to separate the foreground information from the  background noise.
We consider 
two videos datasets: ``Lobby in an office building with switching
on/off lights" and ``Shopping center" available from
\url{http://perception.i2r.a-star.edu.sg/bk\_model/bk\_index.html}.
Here we illustrate how the proposed Krylov method
 can be used to obtain an appropriate dimension of the 
 principle subspace (components) to be
used for background subtraction, and use the approximate principal components obtained from the algorithm
in the application~\cite{candes2011robust}.

The Lobby video contains 1546
frames each of size $160\times128$, and
  the data matrix size  is $1546\times20480$.
Second video is from a  shopping mall with 
   1286    frames  each of  resolution  $320\times256$. So,
the   data   matrix   is   of   size   $1286\times81920$.  
This video contains more activities than Lobby video
with many  people moving  in and  out of  the frames throughout.
The performance of the proposed method
for background subtraction of these video data is shown 
in figure \ref{fig:video1}.

Figure \ref{fig:video1}(four images on the left) 
are results on  a randomly selected frame from the Lobby video. The four images correspond to 
the true frame, low rank approximation (after adding back the mean)
and
the background subtracted image
using the eigenvectors obtained from the proposed Krylov method ($m=10,\sigma=0.1$), and using 
the exact 
eigenvectors, respectively. The images were all mean centered and normalized
to have unit norm. The  approximate dimension estimated  was equal 
 to $1$. 
 The matrix has one very large eigenvalue compared to rest, since 
 the video has very little activities (one/two people moving in and out in 
 few frames).

  Figure
\ref{fig:video1}(C) and (D) are the background subtracted images for 
a randomly selected frame from the Shopping Mall video.
The  approximate dimension  estimated by our method was
 $14$.  This video has more activities and the dimension estimated here 
 is higher than for the Lobby data.
 For more details on these datasets and the use of PCA for foreground detection, we refer~\cite{candes2011robust}.
 We observe that, we can achieve good foreground detection using the 
proposed method.
Also note that, our method does not require forming the covariance matrix for PCA
(in the above two video datasets, $p=20480$ and $81920$, respectively), hence requiring
less storage (such dense covariance matrices would not fit in the memory).
Therefore, this example illustrates how the proposed method
 can be used to simultaneously estimate the dimension of the 
principal subspace and use the approximation obtained for the principal subspace in
PCA and robust PCA applications.


 
\newpage
\onecolumn
 \appendix
 \section{Proofs for the derivation}\label{sec:proofs}
 Here we give the proofs that are missing in the main paper.
 \paragraph{Proof of Proposition~\ref{prop:2}.}
 \begin{proof}
  From proposition~\ref{prop:1},  $\mbf{S}_n$ is a $\sqrt{n}$ consistent estimator of 
$\bSigma$, and we can express  $\mbf{S}_n$  as a perturbation
\[
  \mbf{S}_n =\bSigma +\veps \frac{ \mbf{S}_n -\bSigma}{\veps} =  \bSigma+\veps \mbf{E},
\]
where the perturbation of $\bSigma$ is of the order $1/\sqrt{n}$. That is,
$
 \veps  \mbf{E}=O_p\left(\frac{1}{\sqrt{n}}\right).
$
Then,
\begin{eqnarray*}
  \mbf{S}_n \mbf{U}_q\Lambda_q^{-1}&=& (\bSigma+\veps \mbf{E})\mbf{U}_q\Lambda_q^{-1}\\
  &=& \mbf{U}_q +\veps \mbf{E}\mbf{U}_q\Lambda_q^{-1}.
\end{eqnarray*}
Since $\mbf{U}_q $ has orthogonal columns and is non-random, and also for $\Lambda_q^{-1}$
(diagonal matrix with inverse of the top $q$ eigenvalues)
is bounded since $\lambda_q>\veps$, the second term in the above equation should be
$\veps \mbf{E}\mbf{U}_q\Lambda_q^{-1} = O_p(\veps)$.
Then, we have $\mbf{G}_q=\mbf{U}_q+O_p\left(\frac{1}{\sqrt{n}}\right)$,
i.e., $\mbf{G}_q$ is a $\sqrt{n}$ consistent estimator of $\mbf{U}_q$.
See~\cite{anderson1963asymptotic} for further details.

 \end{proof}

 Proof of the corresponding Corollary:
 
  \begin{proof}
   From proposition~\ref{prop:2}, we have  $\mbf{G}_q=\mbf{U}_q+O_p\left(\frac{1}{\sqrt{n}}\right)$. Then,
   \begin{eqnarray*}
    {\mbf{Q}}_G = \mbf{G}_{p-q}\mbf{G}_{p-q}^T &=& \mbf{I}_p - \mbf{G}_{q}\mbf{G}_{q}^T\\
     &=& \mbf{I}_p -\left[\mbf{U}_q+O_p\left(\frac{1}{\sqrt{n}}\right)\right]\left[\mbf{U}_q+O_p\left(\frac{1}{\sqrt{n}}\right)\right]^T\\
     &=& \mbf{I}_p - \mbf{U}_{q}\mbf{U}_{q}^T + O_p\left(\frac{1}{\sqrt{n}}\right)\\
     &=& \mbf{U}_{p-q}\mbf{U}_{p-q}^T + O_p\left(\frac{1}{\sqrt{n}}\right)
   \end{eqnarray*}

  \end{proof}

  \paragraph{Proof of Proposition~\ref{prop:3}.}
 \begin{proof}
  Using the Corollary, we have
  \begin{eqnarray*}
   \tvec({\mbf{Q}}_G(\mbf{S}_n-\sigma\mbf{I}_p){\mbf{Q}}_G)
   &=& \tvec({\mbf{Q}}_G(\mbf{S}_n-\sigma\mbf{I}_p)(\mbf{Q}_G - \mbf{Q}_U))
   + \tvec({\mbf{Q}}_G(\mbf{S}_n-\sigma\mbf{I}_p){\mbf{Q}}_U)\\
   &=& \tvec({\mbf{Q}}_G(\mbf{S}_n-\sigma\mbf{I}_p){\mbf{Q}}_U) + O_p\left(\frac{1}{\sqrt{n}}\right)\\
    &=&  \tvec(({\mbf{Q}}_G-\mbf{Q}_U)(\mbf{S}_n-\sigma\mbf{I}_p){\mbf{Q}}_U) + \tvec({\mbf{Q}}_U(\mbf{S}_n-\sigma\mbf{I}_p){\mbf{Q}}_U) + O_p\left(\frac{1}{\sqrt{n}}\right)\\
     &=& \tvec({\mbf{Q}}_U(\mbf{S}_n-\sigma\mbf{I}_p){\mbf{Q}}_U) + O_p\left(\frac{1}{\sqrt{n}}\right).\\
 \end{eqnarray*}
Thus, $  \tvec({\mbf{Q}}_G(\mbf{S}_n-\sigma\mbf{I}_p){\mbf{Q}}_G)$ has the same asymptotic distribution
as $  \tvec({\mbf{Q}}_U(\mbf{S}_n-\sigma\mbf{I}_p){\mbf{Q}}_U)$.
We know that the bottom $p-q$ eigenvalues of $\bSigma$ are all $\sigma$. Hence we have
$\mbf{Q}_U\bSigma\mbf{Q}_U=\mbf{Q}_U(\sigma\mbf{I}_p)\mbf{Q}_U$.
So,
we have
\begin{eqnarray*}
 \tvec({\mbf{Q}}_G(\mbf{S}_n-\sigma\mbf{I}_p){\mbf{Q}}_G) &=& \tvec({\mbf{Q}}_U(\mbf{S}_n-\bSigma){\mbf{Q}}_U)
+ O_p\left(\frac{1}{\sqrt{n}}\right)\\
 &=& (\mbf{Q}_U \otimes\mbf{Q}_U) \tvec(\mbf{S}_n-\bSigma)  + O_p\left(\frac{1}{\sqrt{n}}\right).
\end{eqnarray*}
Thus, in terms of the distribution, we have from above,
\[
 \sqrt{n}\mathbb{E}\{\tvec({\mbf{Q}}_G(\mbf{S}_n-\sigma\mbf{I}_p){\mbf{Q}}_G)\} = 
  (\mbf{Q}_U \otimes\mbf{Q}_U)  \mathbb{E}\{\sqrt{n}\tvec(\mbf{S}_n-\bSigma) \} = 0
\]
and
\begin{eqnarray*}
 cov\{\sqrt{n}\tvec({\mbf{Q}}_G(\mbf{S}_n-\sigma\mbf{I}_p){\mbf{Q}}_G)\}& = &
  (\mbf{Q}_U \otimes\mbf{Q}_U)  cov\{\sqrt{n}\tvec(\mbf{S}_n-\bSigma) \} (\mbf{Q}_U \otimes\mbf{Q}_U)\\
  & = & (\mbf{Q}_U \otimes\mbf{Q}_U) \bOmega (\mbf{Q}_U \otimes\mbf{Q}_U) .
\end{eqnarray*}
 \end{proof}
  
  \paragraph{Proof of Theorem~\ref{theo:criteria}.}
 \begin{proof}
 A model selection criterion takes the form 
 \[
  IC(k)= L(n,k)- \mathbb{E}(L(n,k)),
 \]
as $n\longrightarrow\infty$, for $L(n,k)\rightarrow_d \chi^2$ distribution~\cite{anderson2003introduction}. In our case, from Lemma~\ref{lemm:1}, we have
\[
 L(n,k) = \sum_{i=1}^{\eta} \mu_i \chi^2_{(1)},
\]
where $\mu_i$ are the eigenvalues of
$\frac{1}{2\sigma^2}(\mbf{Q}_G \otimes\mbf{Q}_G)(\mbf{S}_n\otimes\mbf{S}_n) (\mbf{Q}_G \otimes\mbf{Q}_G)$,
an estimate of $\frac{1}{2\sigma^2}\hat{\bOmega}$, from Proposition~\ref{prop:3},
the asymptotic covariance matrix of $\sqrt{\frac{n}{2\sigma^2}}\tvec({\mbf{Q}}_G(\mbf{S}_n-\sigma\mbf{I}_p){\mbf{Q}}_G)$,
and square of Gaussian is $\chi^2_{(1)}$.
To compute an approximation to the mean of the statistic, we use the following Gamma approximation:
\begin{eqnarray*}
 \sum_{i=1}^{\eta} \mu_i \chi^2_{(1)} &= & \sum_{i=1}^{\eta} \mu_i \Gamma\left(\frac{1}{2},2\right)
  =  \sum_{i=1}^{\eta} \Gamma\left(\frac{1}{2},2\mu_i \right)\\
   &= & \sum_{i=1}^{\eta} \Gamma(\kappa,\theta_i )
    \simeq  \Gamma(K,\Theta )\\
\end{eqnarray*}
where
\[
 \kappa=\frac{1}{2}, \theta_i=2\mu_i, K=\frac{(\sum_i\kappa\theta_i)^2}{\sum_i\theta_i^2\kappa} \text{ and }
 \Theta= \frac{\sum_i\kappa\theta_i}{K}
\]
and the mean of the asymptotic approximation of $L$ is given by  
$\mathbb{E}(L(n,k))=K\Theta$.
Hence, in our case, 
\[
 \mathbb{E}(L(n,k))= \sum_{i=1}^{\eta} \kappa\theta_i= \sum_{i=1}^{\eta} \mu_i =
\sum_{i,j=k+1;i\neq j}^{p}  \frac{\ell_{i}*\ell_{j}}{2\sigma^2},
\]
where $\{\ell_i\}_{i=1}^p$ are the eigenvalues of the sample covariance matrix $\mbf{S}_n$
and the last equality is from the property of Kronecker products as seen in the proof of  Lemma~\ref{lemm:1}.

Note that, asymptotically $\ell_i\rightarrow\sigma$,the noise variance, for $i>q$ as $n\rightarrow\infty$.
Hence, asymptotically
\[
 \mathbb{E}(L(n,k))\rightarrow\eta=\frac{(p-k)(p-k-1)}{2}.
\]
Hence, we use the criterion in~\eqref{eq:crit} for model selection, i.e., for the dimension estimation
of the principal subspace.
\end{proof}

\begin{wrapfigure}[12]{l}{0.4\textwidth}
  \begin{center}
\includegraphics[width=0.3\textwidth,,trim={2cm 2cm 2cm 1.6cm}]{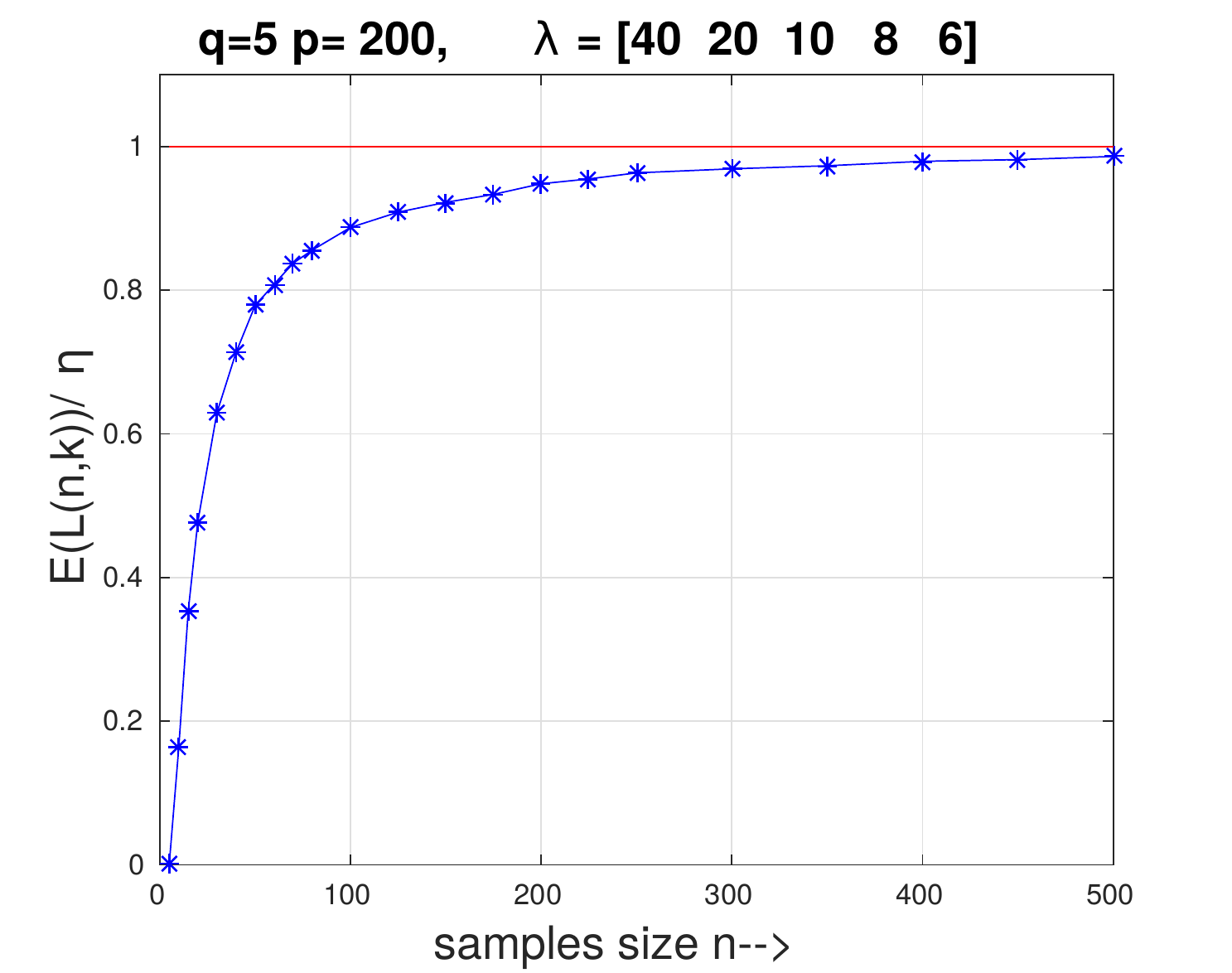}
  \end{center}
  \end{wrapfigure}
  
The figure on the left plots the ratio 
\[
\frac{\mathbb{E}(L(n,q))}{\eta}=\frac{\sum_{i,j=k+1;i\neq j}^{p}  
\frac{\ell_{i}*\ell_{j}}{2\sigma^2}}{\frac{(p-q)(p-q-1)}{2}}
\] 
as a function of 
the number of samples $n$ for a  small simulation with $p=200,q=5$ (similar to the experiment
in Figure~\ref{fig:1}).
The true covariance matrix from which the data is sampled 
has top $q=5$ eigenvalues of magnitude listed in the figure and 
the noise level was $\sigma=1.2$. We plot the average of the ratio over 30 trials.
We note that the mean $\mathbb{E}(L(n,q))$ quickly approaches 
the degree of freedom $\eta$, showing that the quantity $L(n,q)$ indeed has
$\chi^2_\eta$ distribution for large enough $n$. 
Thus, Lemma~\ref{lemm:1} and Theorem~\ref{theo:criteria} hold in practice too.
In section~\ref{sec:expt} of the main paper and below, we present several numerical experiments to illustrate
the performance of the proposed method.

  \section{Proofs for the analysis}\label{sec:proofs2}
  
  \paragraph{Proof of Theorem~\ref{theo:1}}
 \begin{proof}
 In order to prove the strong consistency of 
 \[
  \hat{k}=\arg\min_kIC(k), 
  \]
we first consider that $\hat{k}>k_0$, then
\begin{eqnarray*}
IC(\hat{k})-IC(k_0)& =& \frac{n}{2\sigma^2}\left(\sum_{i=\hat{k}+1}^p(\ell_i-\sigma)^2 - \sum_{i=k_0+1}^p(\ell_i-\sigma)^2\right)
-C_n\left(\frac{(p-\hat{k})(p-\hat{k}-1)}{2}-\frac{(p-k_0)(p-k_0-1)}{2}\right)\\
&=& -\frac{n}{2\sigma^2}\sum_{i=k_0+1}^{\hat{k}}(\ell_i-\sigma)^2-C_n\left(\frac{(\hat{k}-k_0)(\hat{k}+k_0-2p+1)}{2}\right)\\
\frac{IC(\hat{k})-IC(k_0)}{n}&=&-\frac{1}{2\sigma^2}\sum_{i=k_0+1}^{\hat{k}}(\lambda_i-\sigma)^2-
\frac{C_n}{n}\left(\frac{(\hat{k}-k_0)(\hat{k}+k_0-2p+1)}{2}\right)+O\left(\sqrt{\frac{\log\log n}{n}}\right),
\end{eqnarray*}
since $\ell_i=\lambda_i+O\left(\sqrt{\tfrac{\log\log n}{n}}\right)$ from
the law of iterated logarithm~\cite{muirhead2009aspects}.
The last two terms in the RHS of the above equation go to zero as $n$ tends to infinity and $\lambda_i>0$, hence we have
\[
 IC(\hat{k})-IC(k_0)<0 \text{ for all large $n$ a.s.}
\]
Next, for $\hat{k}<k_1$, we have 
\begin{eqnarray*}
IC(\hat{k})-IC(k_1)& =& \frac{n}{2\sigma^2}\left(\sum_{i=\hat{k}+1}^p(\ell_i-\sigma)^2 - \sum_{i=k_1+1}^p(\ell_i-\sigma)^2\right)
-C_n\left(\frac{(p-\hat{k})(p-\hat{k}-1)}{2}-\frac{(p-k_1)(p-k_1-1)}{2}\right)\\
&=& \frac{n}{2\sigma^2}(k_1-\hat{k})O\left({\frac{\log\log n}{n}}\right)-C_n\left(\frac{(\hat{k}-k_1)(\hat{k}+k_1-2p+1)}{2}\right)\\
\frac{IC(\hat{k})-IC(k_1)}{C_n}&=&-\frac{(\hat{k}-k_1)(\hat{k}+k_1-2p+1)}{2}+
\frac{(k_1-\hat{k})}{2\sigma^2}.\frac{O(\log\log n)}{C_n}.
\end{eqnarray*}
Since,  $\ell_i=\lambda_i+O\left(\sqrt{\tfrac{\log\log n}{n}}\right)$ and for all $i>\hat{k},\lambda_i=\sigma$.
Again, the second term in the RHS of the above equation goes to zero due the the property of $C_n$. As, $\hat{k}<k_1$ and 
$\{\hat{k},k_1\}<p$, the first term is always negative. Hence, we again have 
\[
 IC(\hat{k})-IC(k_1)<0 \text{ for all large $n$ a.s.}
\]
\end{proof}

  \paragraph{Proof of Corollary~\ref{corr:2}}

  \begin{proof}
  For the eigenvalues $\theta_i$ computed in  Algorithm~\ref{alg:algo1}, we have from Lemma~\ref{lemm:2},
  $$\ell_i-\eps\ell_{k+1}\leq \theta_i\leq \ell_i, \: i=1,\ldots,k.$$
  Hence, we have \[
  IC(k)\leq \frac{n}{2\sigma^2}\left(\|\mbf{S}_n-\sigma\mbf{I}_p\|_F^2-\sum_{i=1}^k(\ell_i-\eps \ell_{k+1} -\sigma)^2 \right)
   -C_n\frac{(p-k)(p-k-1)}{2}.
   \]
   For the first case when $\hat{k}>k_0$, ignoring the terms that go to zero asymptotically, we will have:
  \begin{eqnarray*}
    \frac{IC(\hat{k})-IC(k_0)}{n}&\leq &\frac{1}{2\sigma^2}\left(\sum_{i=1}^{k_0}(\ell_i-\eps \ell_{k_0+1} -\sigma)^2-
    \sum_{i=1}^{\hat{k}}(\ell_i-\eps \ell_{\hat{k}+1} -\sigma)^2\right)\\
    &=&\frac{1}{2\sigma^2}\left(\sum_{i=1}^{k_0}\left((\ell_i-\eps \ell_{k_0+1} -\sigma)^2-(\ell_i-\eps \ell_{\hat{k}+1} -\sigma)^2\right)
    -    \sum_{i=k_0+1}^{\hat{k}}(\ell_i-\eps \ell_{\hat{k}+1} -\sigma)^2\right).
  \end{eqnarray*}
For $\eps<1$, note that both terms in  RHS is always negative since $\ell_{k_0+1}> \ell_{\hat{k}+1}$. 
Hence $IC(\hat{k})-IC(k_0)<0 $ for eigenvalues computed by the Krylov method.

Next, for the case $\hat{k}<k_1$, the  term in $\tfrac{IC(\hat{k})-IC(k_1)}{C_n}$ which is
neither negative nor goes to zero is
  \begin{eqnarray*}
 \frac{IC(\hat{k})-IC(k_1)}{C_n}&\leq &\frac{n}{2\sigma^2C_n}\left(\sum_{i=\hat{k}+1}^{k_1}(\ell_i-\eps \ell_{k_1+1} -\sigma)^2\right)\\
 &=& \frac{n}{2\sigma^2C_n}(k_1-\hat{k})\left((1-\eps)(\sigma+O\left(\sqrt{\tfrac{\log\log n}{n}}\right))-\sigma \right)^2.
 \end{eqnarray*}
Hence, if we replace $\sigma$ in the algorithm by $(1-\eps)\sigma$, this term goes to zero and 
we will have $IC(\hat{k})-IC(k_1)<0 $.

\end{proof}
\section{Additional Numerical Results}

 In section~\ref{sec:expt} of the main paper, we presented several numerical experiments to illustrate
the performance of the proposed method in applications. Here, we present few additional experimental results.

 \begin{figure*}[tb!]
 \begin{center}
\includegraphics[width=0.32\textwidth]{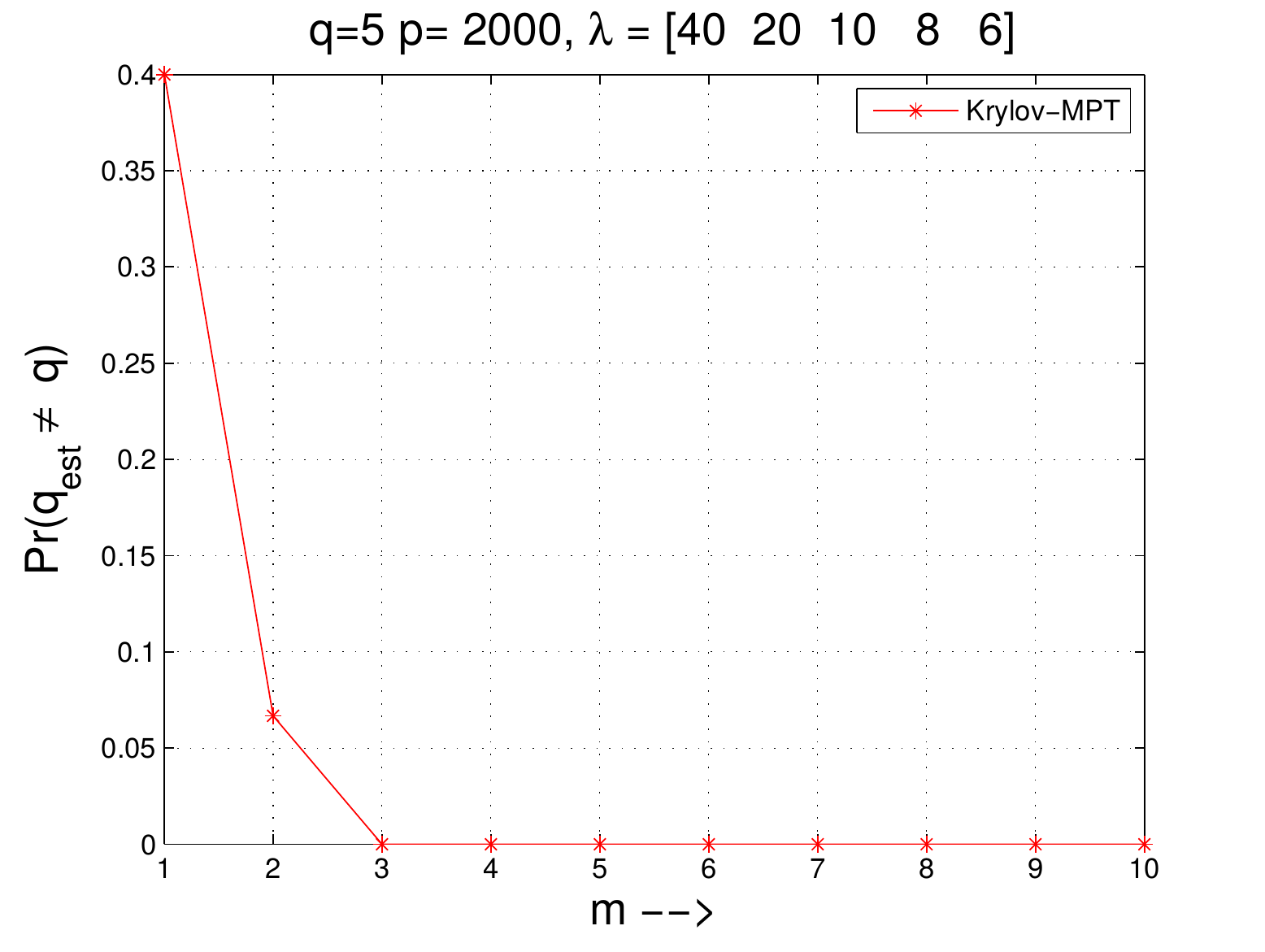}
\includegraphics[width=0.32\textwidth]{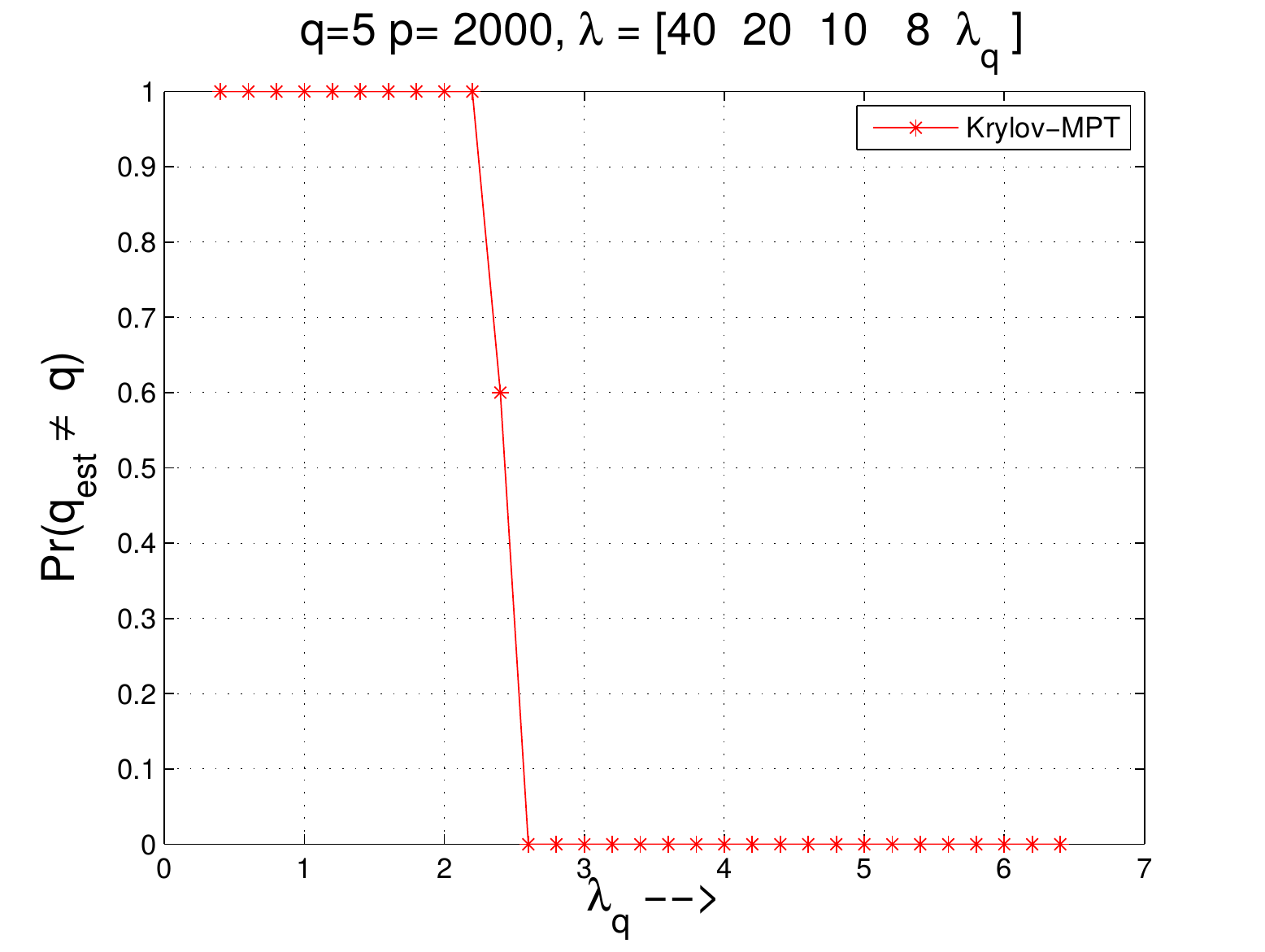}
\includegraphics[width=0.32\textwidth]{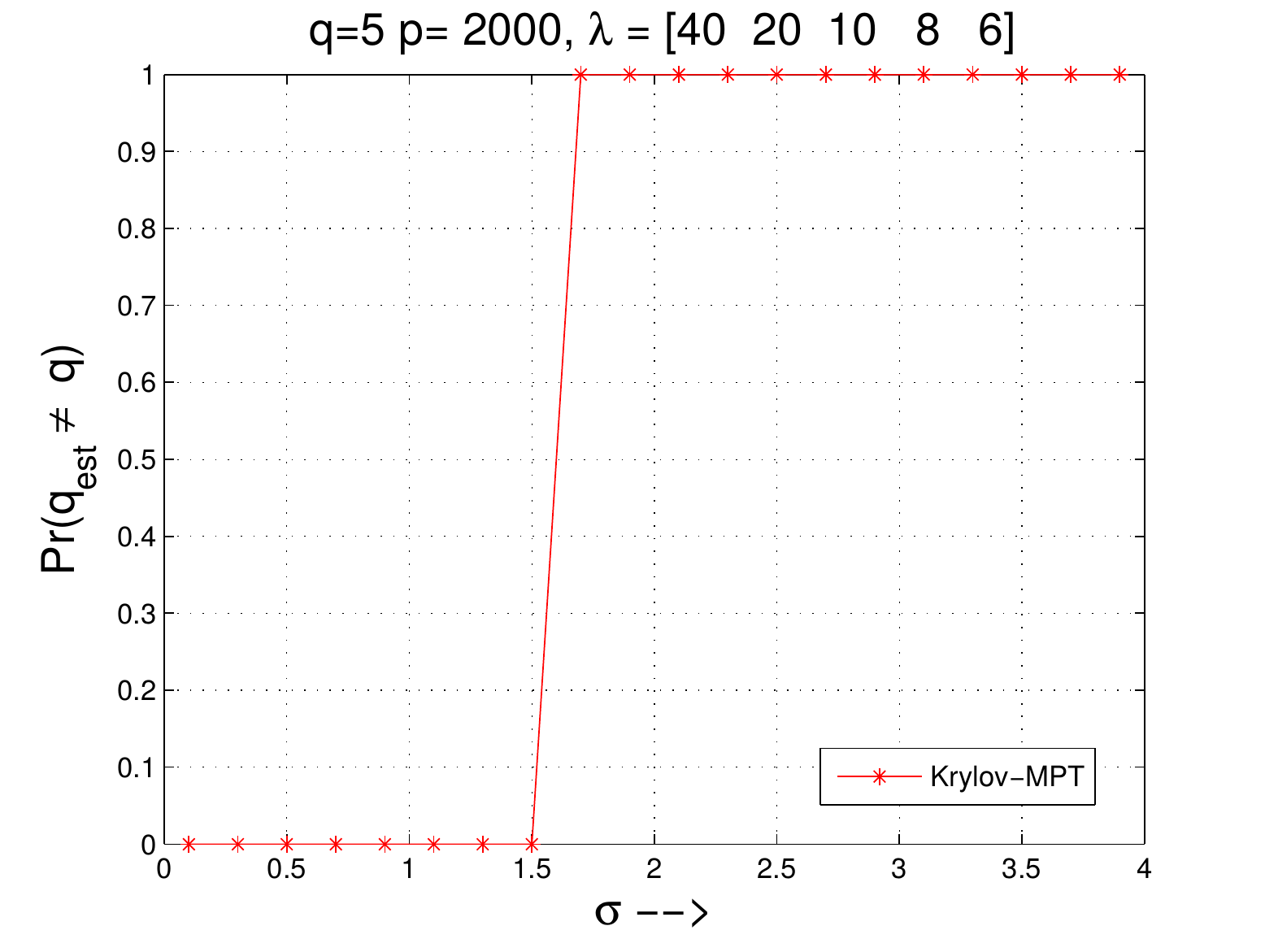}
\vskip -0.1in
 \caption{Signal detection using the Krylov method: Detection as a function of: 
number of Lanczos steps~$m$ (left),  signal strength 
 ($\ell_q$ eigenvalue), and (right) the noise level $\sigma$.}\label{fig:2}
 \end{center}
\end{figure*}
 
\paragraph{Krylov subspace method:}
In the the main paper, for the number of signal detection experiments, we used the exact eigenvalues 
 of the covariance matrices (computed using \texttt{eig} function in Matlab) for the dimension 
estimation using the three compared methods (MDL and RMT require all of the eigenvalues).
Here, we illustrate how the proposed Krylov subspace based algorithm~\ref{alg:algo1} performs
for the dimension estimation. We consider the same signal detection problem as above
(same Gaussian model as Fig.~1). 
The first plot in figure~\ref{fig:2} give the performance of the algorithm 
as a function of the number of Lanczos steps $m$. The parameters were chosen to be 
$p=2000,n=2500,\sigma=1.1$.
We know the relation between the error $\eps$ in the eigenvalue estimation
by the Lanczos algorithm and the number of Lanczos steps $m$ from Lemma~\ref{lemm:2}.
Hence, increasing $m$ is equivalent to decreasing $\eps$.
We see that for a very few Lanczos steps $m\geq4$, we get accurate results.
This is because, it is well-known that
the top eigenvalues computed by Lanczos algorithm converges fast~\cite{Saad-book3}.
This superior performance  of the Lanczos algorithm was observed in~\cite{xu1994fast2} 
as well for a similar Gaussian signal detection model.

In the second and third plots, we plot the  performance of the Krylov subspace method for signal detection
as a function of the signal strength (magnitude of $\lambda_q$ in the middle) and the
noise level $\sigma$ (right), 
with $p=2000, n=2500,m=5$. 
We observed that, our Algorithm~\ref{alg:algo1}, for $m\geq4$, 
performs very well and replicates the results we obtained by the proposed method 
with exact eigenvalues of the sample covariance matrix
 (reported in Figure~\ref{fig:1}).

 \begin{figure*}[tb!]
 \begin{center}
\includegraphics[width=0.24\textwidth]{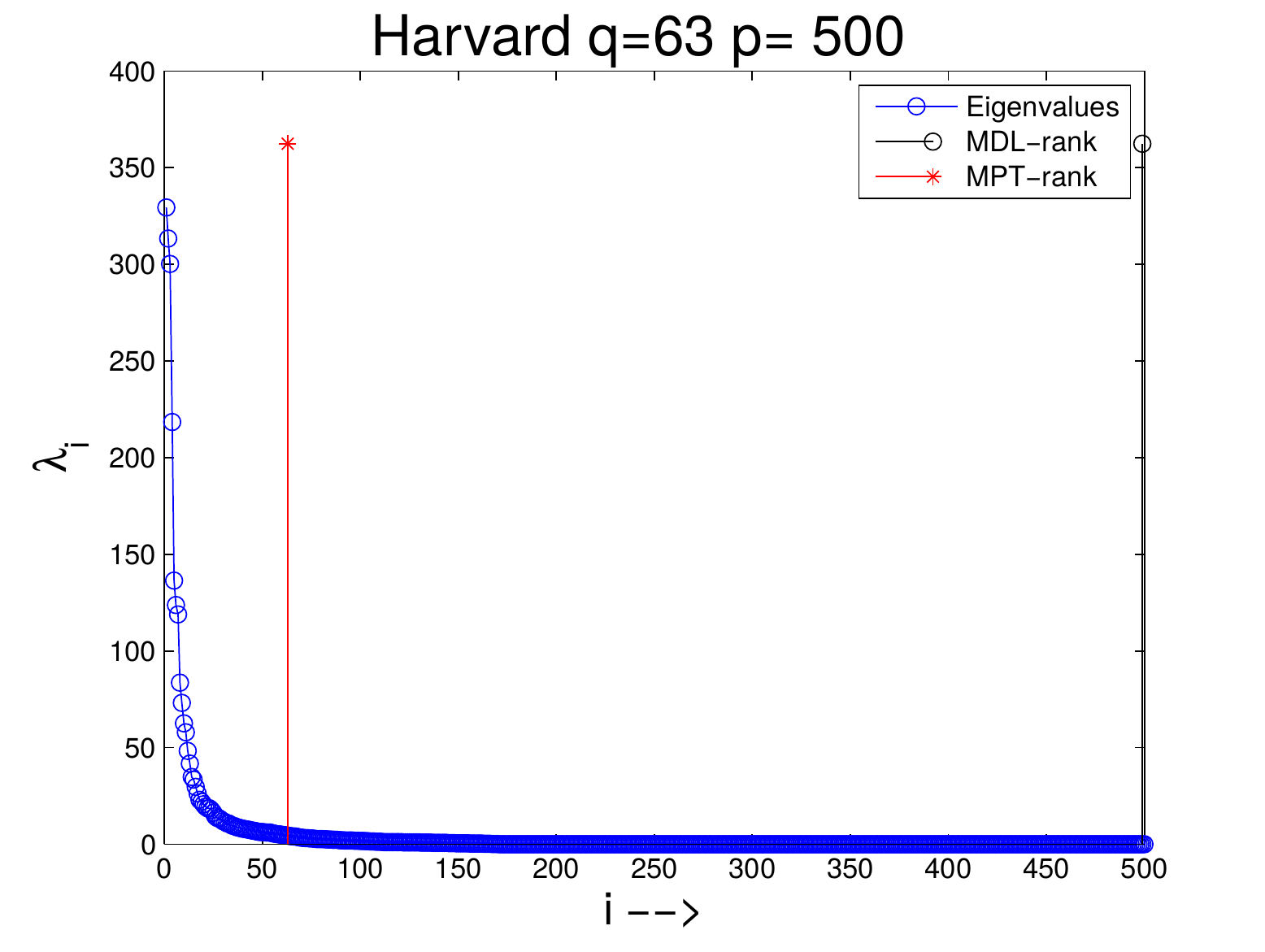}
\includegraphics[width=0.24\textwidth]{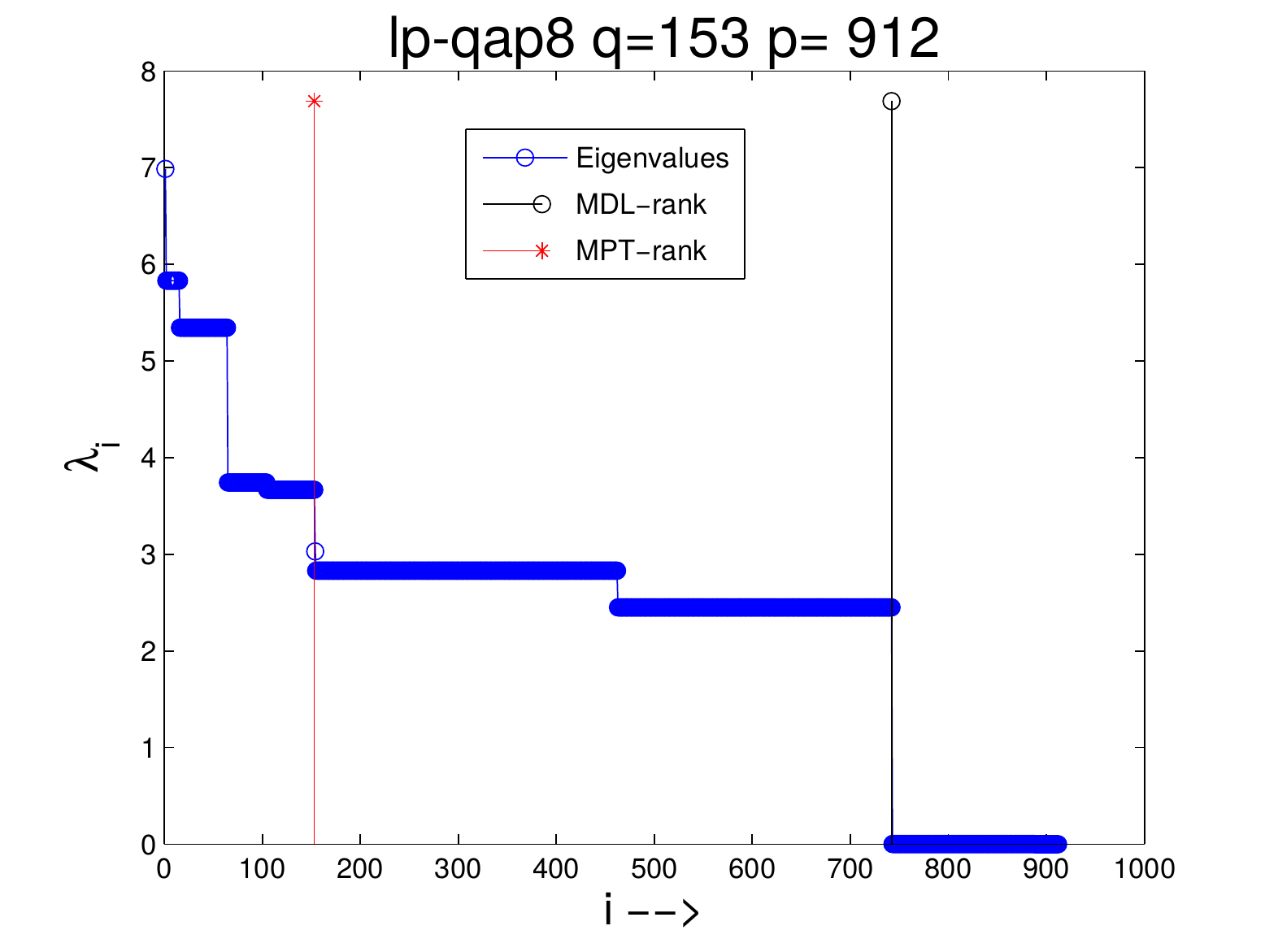}
\includegraphics[width=0.24\textwidth]{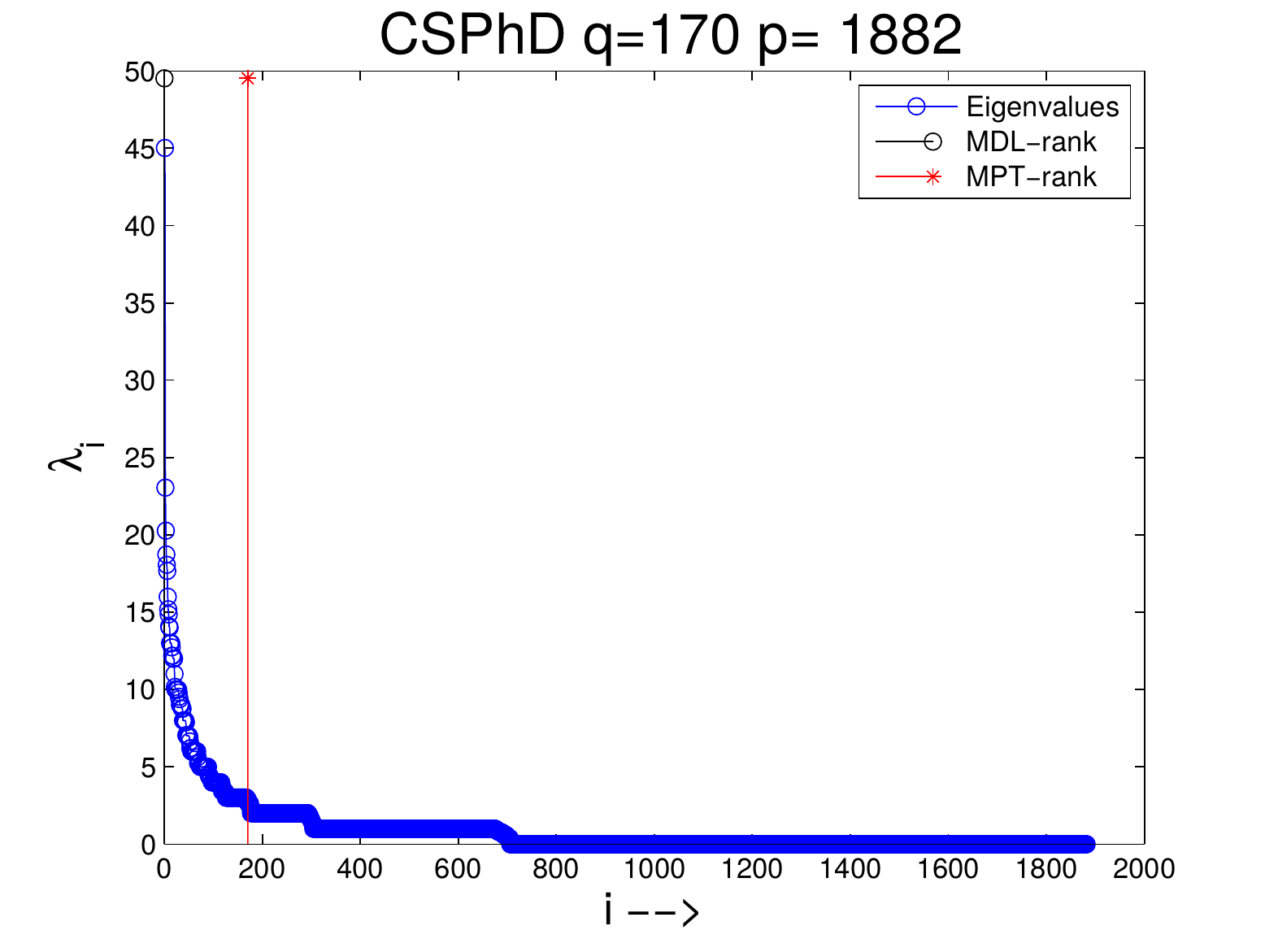}
\includegraphics[width=0.24\textwidth]{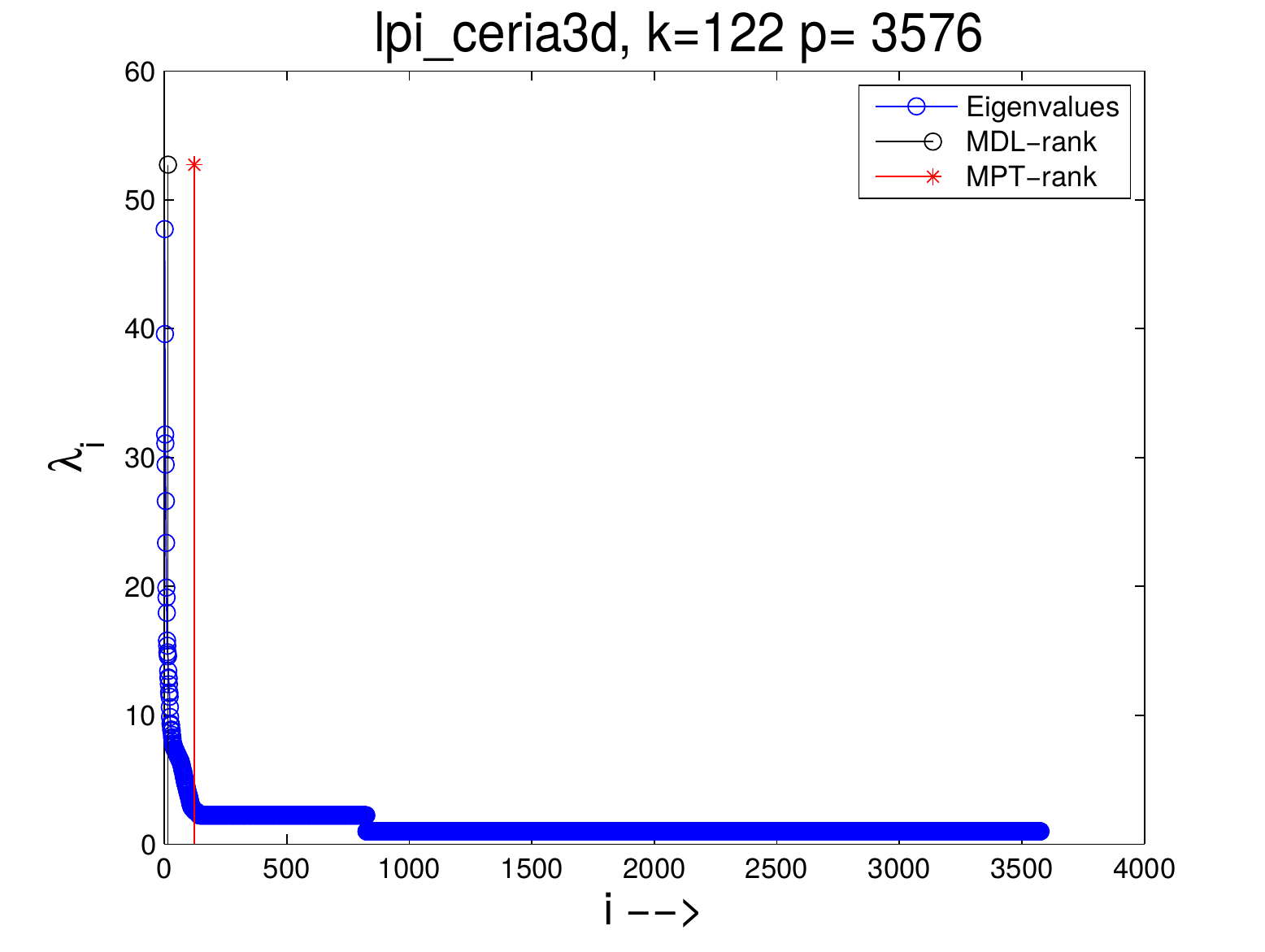}
\includegraphics[width=0.24\textwidth]{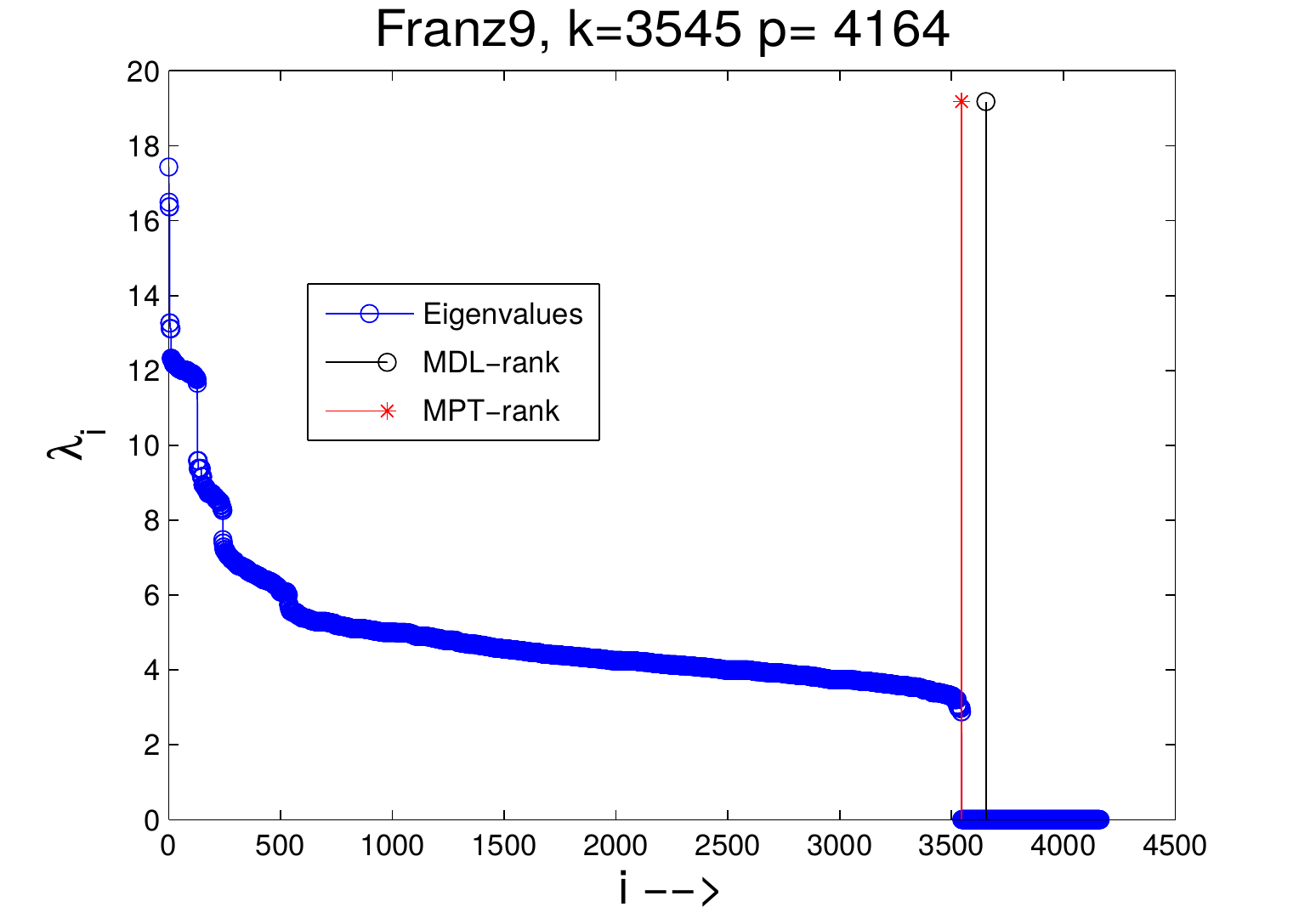}
\includegraphics[width=0.24\textwidth]{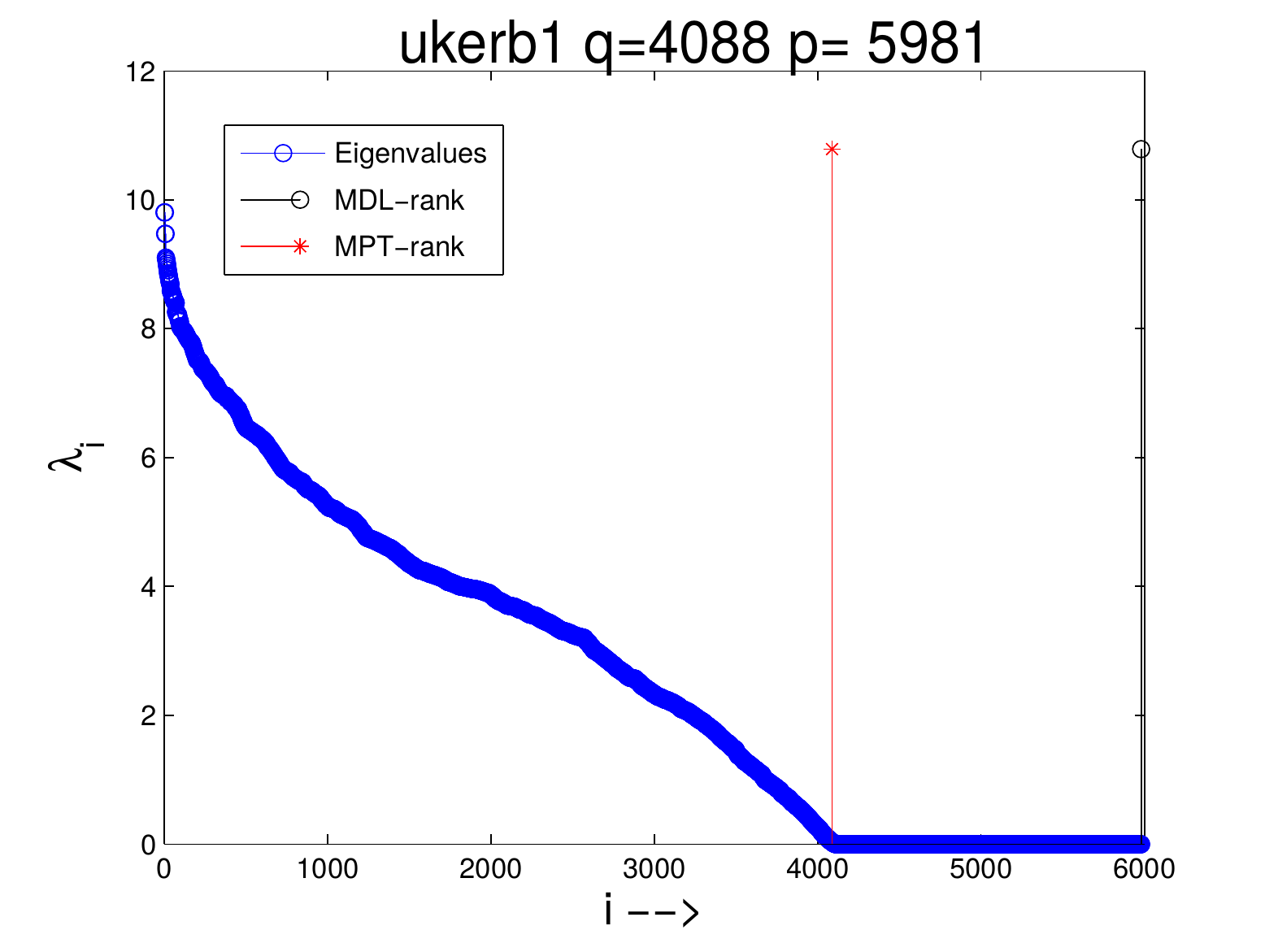}
\includegraphics[width=0.24\textwidth]{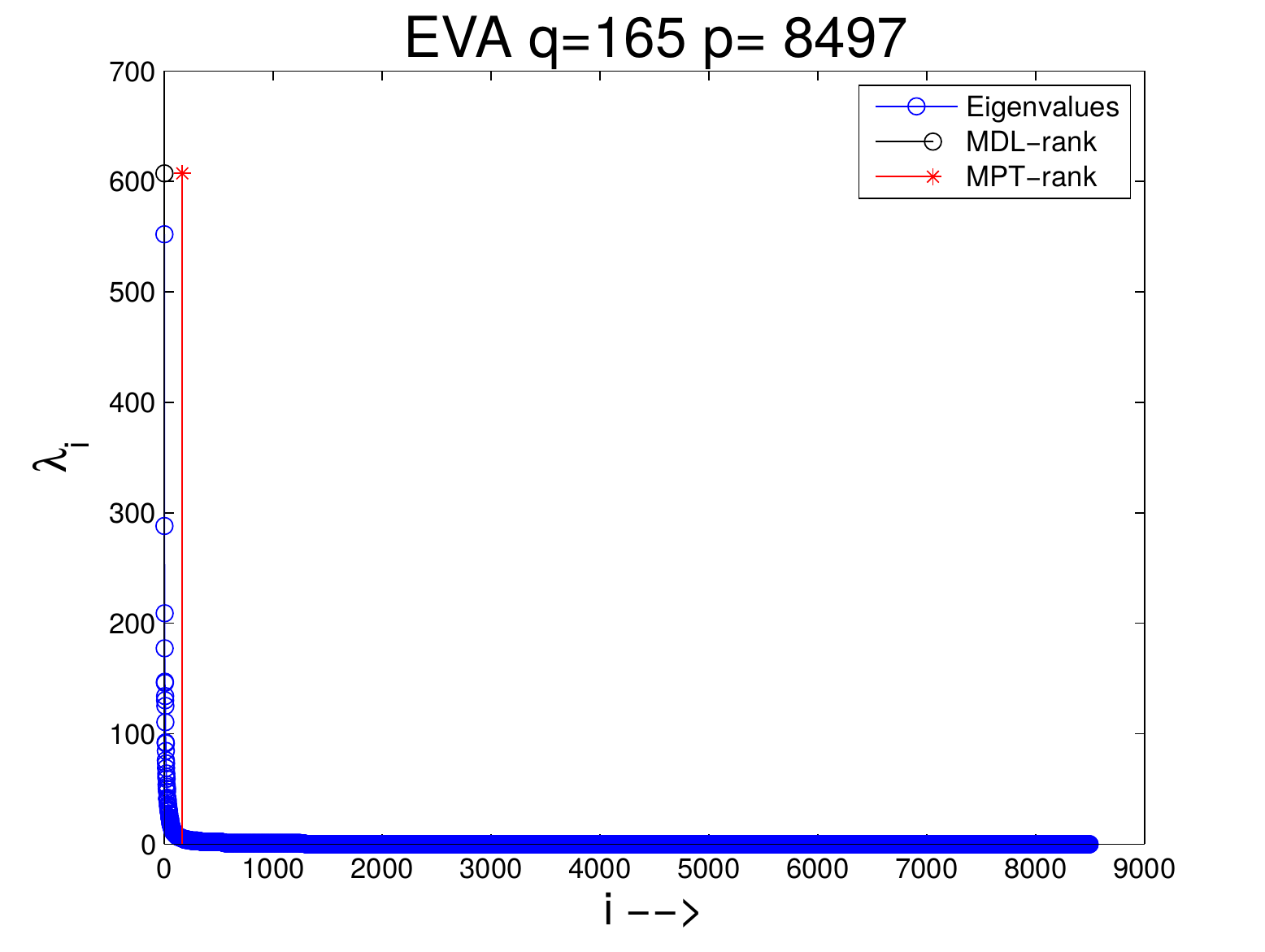}
\includegraphics[width=0.24\textwidth]{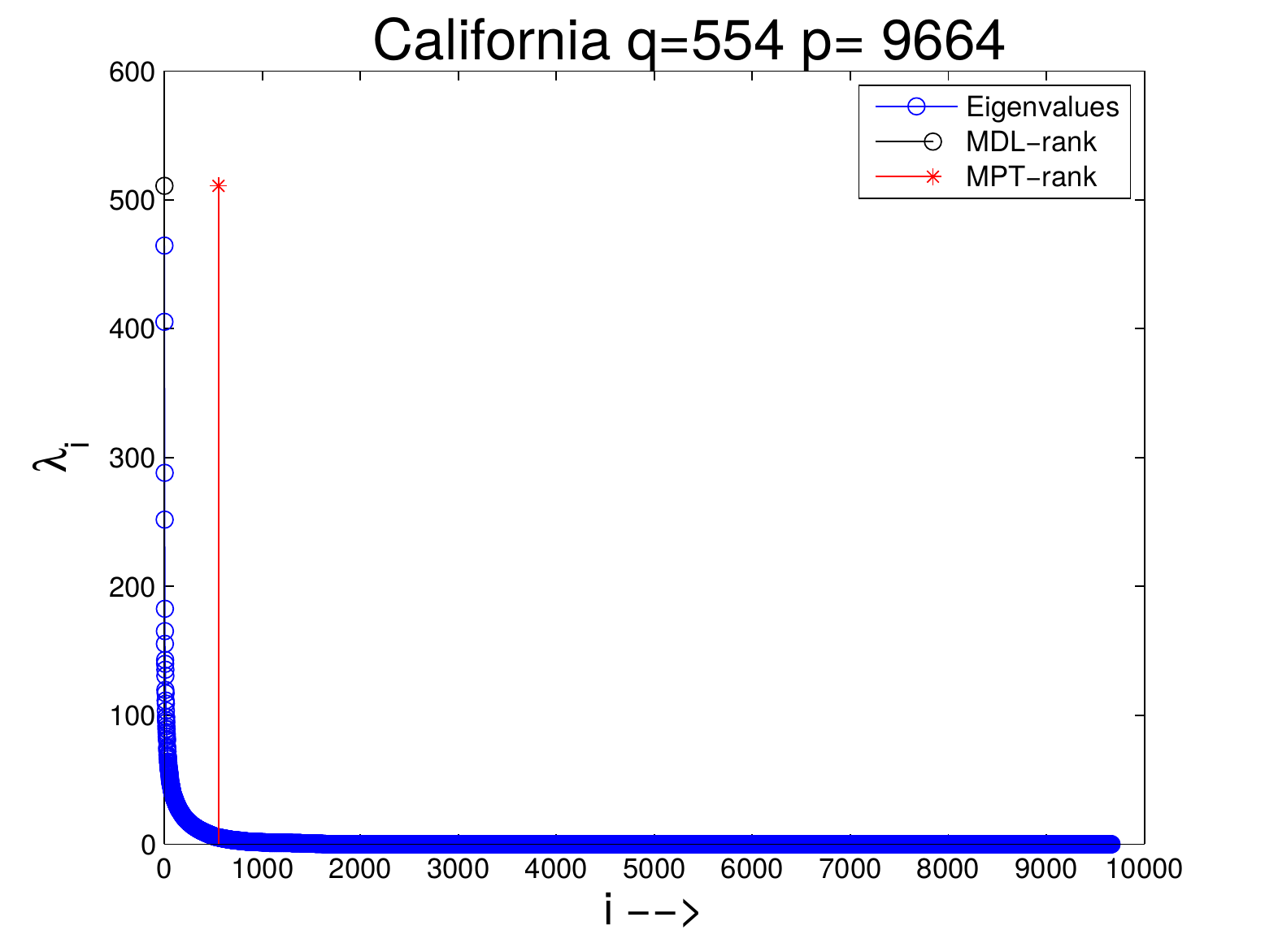}
\includegraphics[width=0.24\textwidth]{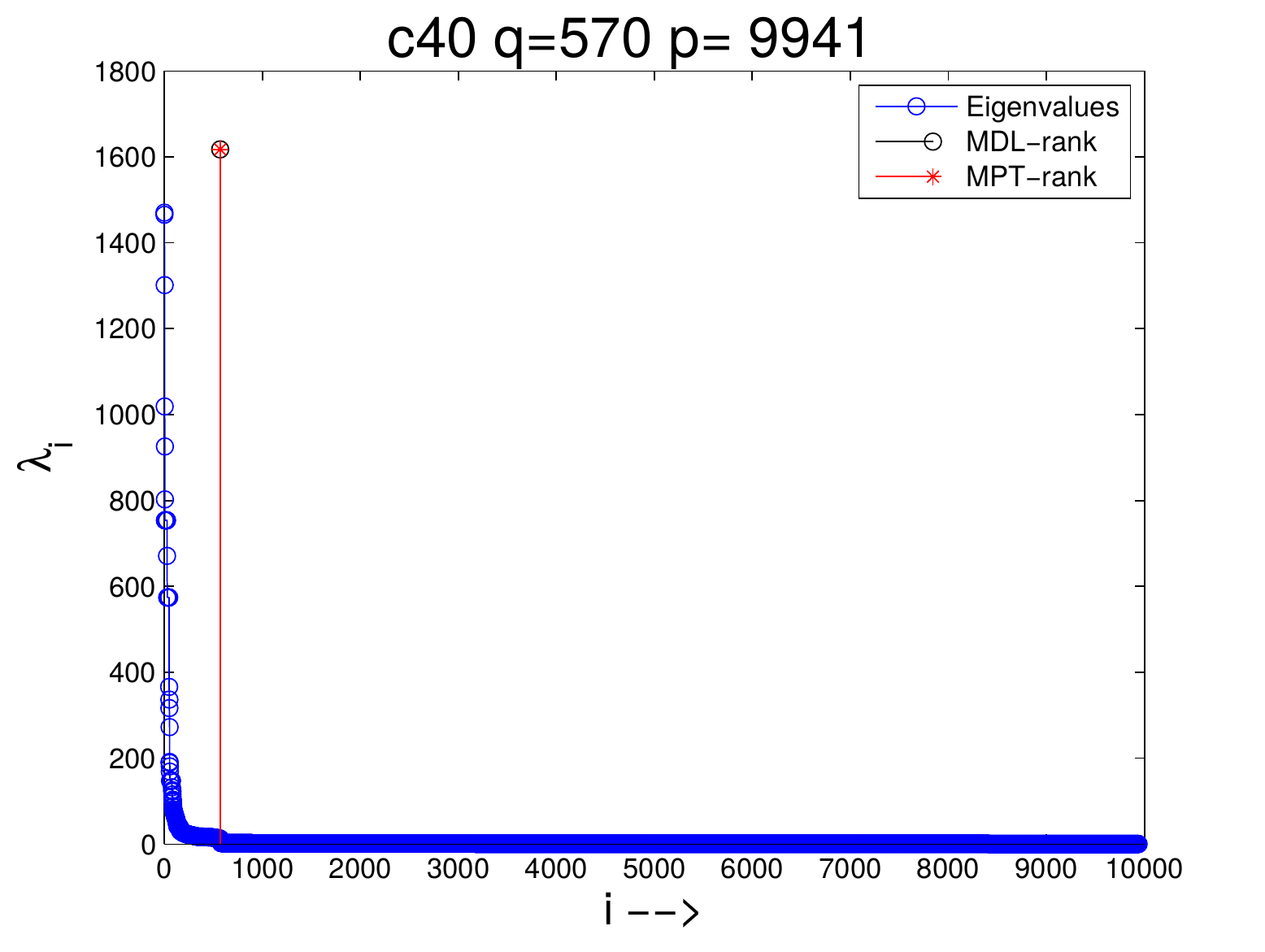}
\includegraphics[width=0.24\textwidth]{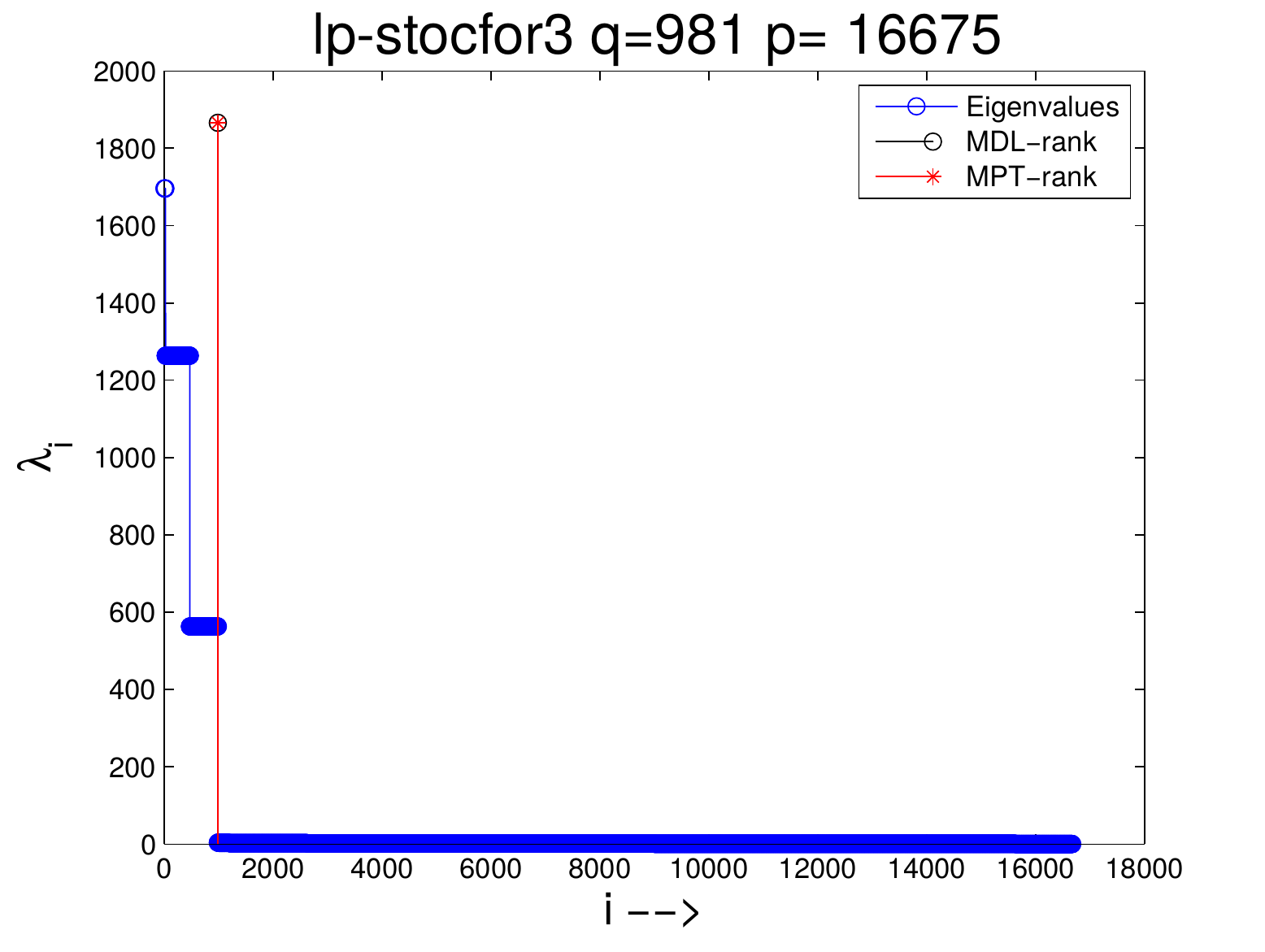}
\includegraphics[width=0.24\textwidth]{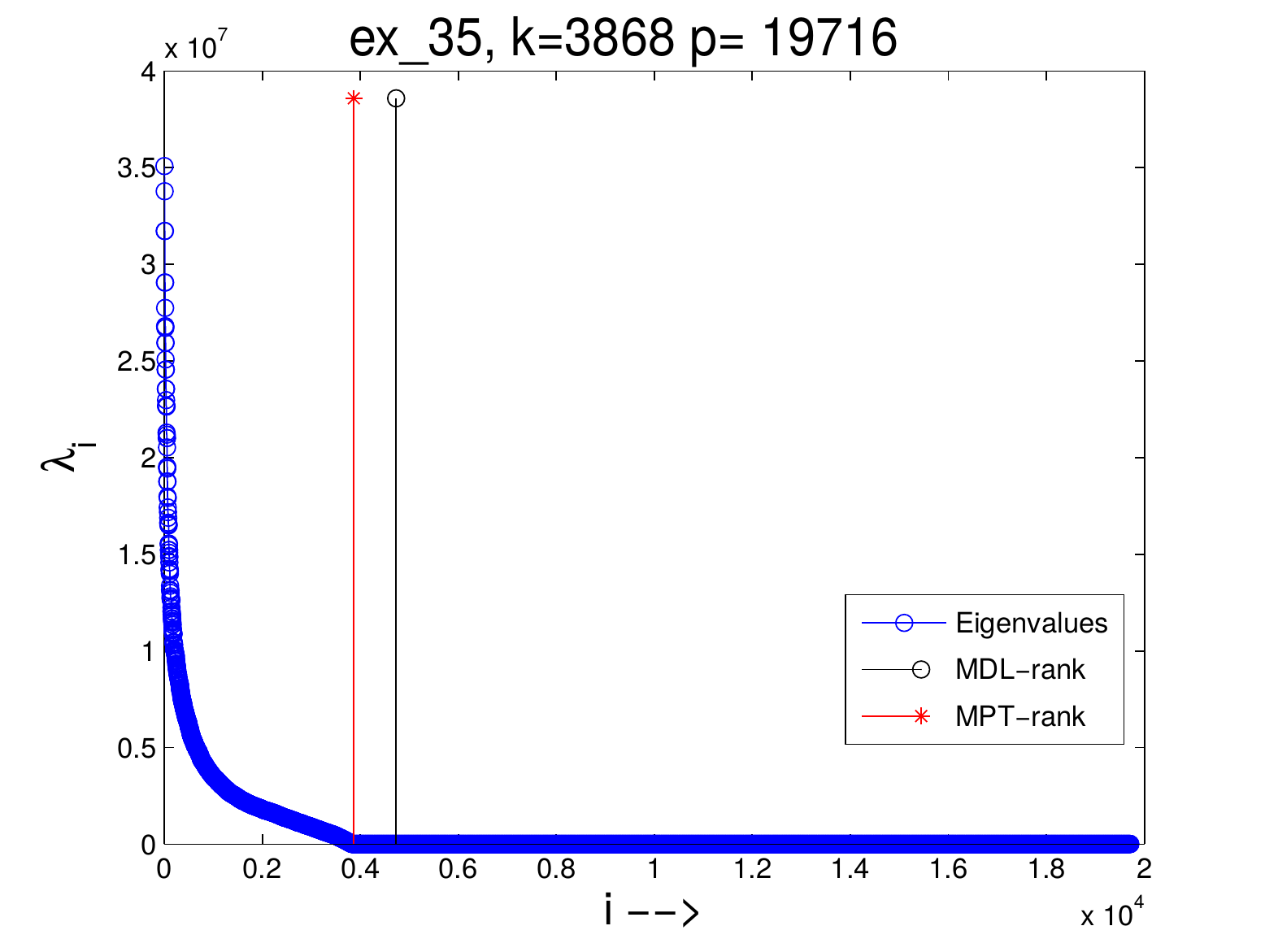}
\includegraphics[width=0.24\textwidth]{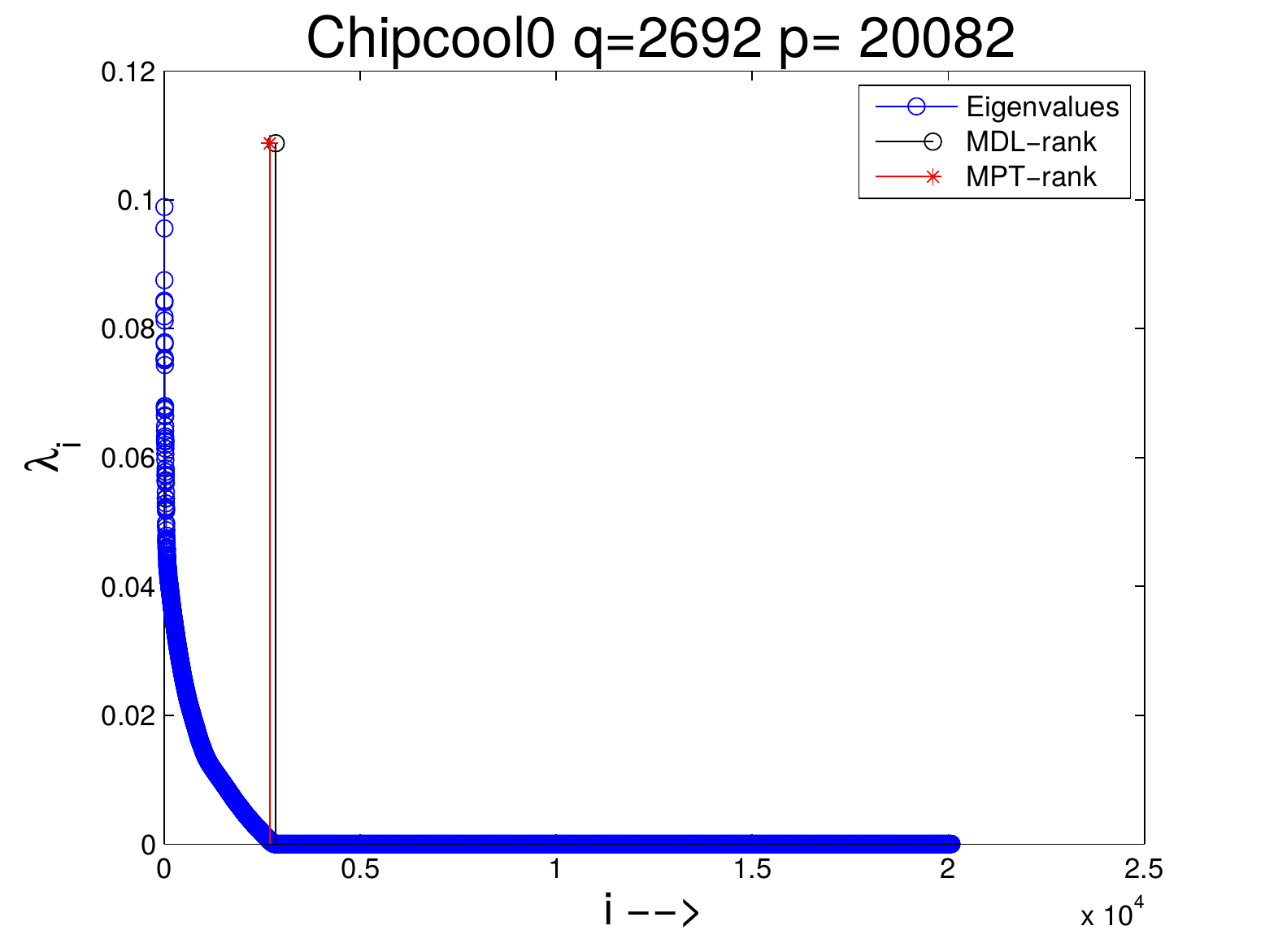}
\vskip -0.1in
 \caption{Numerical rank estimation of data matrices 
 by the proposed method (MPT) and MDL, along with the actual spectrum.}\label{fig:3}
 \end{center}
 \vskip -0.1in
\end{figure*}

\paragraph{Data Matrices:}
In Table~\ref{tab:1} of the main paper, we saw the performance of the proposed algorithm
on few sparse data matrices. The following results give us more insight into the method's performance.
Figure~\ref{fig:3} presents the spectrum of twelve  matrices obtained from the SuiteSparse database with low numerical rank and gap in the spectrum,
along with the 
rank estimated by the the proposed method (MPT) as a red (star) line 
and MDL in black (circle).
We chose $C_n=\log n$ in all cases and  $\sigma=1$ (except chipcool0 where $\sigma=0.01$ was chosen). 
The matrix name, size $p$ and the actual numerical rank $q$ 
(based on the gap) are given in the title of each plot. 
We note that the proposed method gives god solution for almost all examples
except one case 
(lp-qap8, second plot,  the method chooses a different gap in the spectrum for $\sigma=1$).
The MDL method fails  in a few examples and is slightly 
off in a couple more examples.
The matrix lpiceria3d (fourth plot/1st row 1st column) is interesting because the matrix 
has two distinct eigen-gaps close to zero. Our method 
selects the first one. These set of experiments show that 
the proposed method performs very well (determines the rank based on the spectral gap) for
 general data matrices too, where the distribution assumptions do not hold.

\end{document}